\documentclass{amsart}
\usepackage[english]{babel}
\usepackage{amsfonts, amsmath, amsthm, amssymb,amscd,indentfirst}

\usepackage{color}
\usepackage{MnSymbol}

\usepackage{esint}
\usepackage{graphicx}
\usepackage{epstopdf}
\usepackage{amsmath,amssymb,latexsym,indentfirst}
\usepackage{times}
\usepackage{palatino}
\usepackage{marginnote}

\def\R{\mathbb{R}}

\def\d{\partial}

\def\ba{\begin{align}}
\def\ea{\end{align}}
\def\bp{\begin{proof}}
	\def\ep{\end{proof}}

\newtheorem{theorem}{Theorem}[section]

\newtheorem{proposition}[theorem]{Proposition}
\newtheorem{lemma}[theorem]{Lemma}
\newtheorem{definition}[theorem]{Definition}

\newtheorem{corollary}[theorem]{Corollary}

\newtheorem{conjecture}[theorem]{Conjecture}

\newtheorem{example}[theorem]{Example}
\newtheorem{remark}[theorem]{Remark}

\def\R{\mathbb{R}}

\def\d{\partial}

\def\bp{\begin{proof}}
\def\ep{\end{proof}}

\def\R{{\cal R}}

\def\Dc{{\mathcal D}}

\def\Lc{{\mathcal L}}

\def\R{\mathbb{R}}

\def\d{\partial}

\begin{document}

\title[The mass of an asymptotically hyperbolic manifold]{The mass of an asymptotically hyperbolic manifold with a non-compact boundary}

\author{S\'{e}rgio Almaraz}
\address{Universidade Federal Fluminense (UFF), Instituto de Matem\'{a}tica, Campus do Gragoat\'a\\
Rua Prof. Marcos Waldemar de Freitas, s/n, bloco H, 24210-201, Niter\'{o}i, RJ, Brazil.}
              \email{almaraz@vm.uff.br}
\author{Levi Lopes de Lima}
\address{Universidade Federal do Cear\'a (UFC),
Departamento de Matem\'{a}tica, Campus do Pici, Av. Humberto Monte, s/n, Bloco 914, 60455-760,
Fortaleza, CE, Brazil.}
\email{levi@mat.ufc.br}
%\thanks{Both authors were partially supported by  CNPq/Brazil grants.}
\thanks{S. Almaraz has been partially suported by  CNPq/Brazil grant 
	309007/2016-0 and CAPES/Brazil grant 88881.169802/2018-01, and L. de Lima has been partially supported by CNPq/Brazil grant
	311258/2014-0. Both authors have been partially suported by FUNCAP/CNPq/PRONEX grant 00068.01.00/15.}

\begin{abstract}
We define a mass-type invariant for asymptotically hyperbolic manifolds with a non-compact boundary which are modelled at infinity on the hyperbolic half-space and prove a sharp positive mass inequality in the spin case under suitable dominant energy conditions. As an application we show that any such manifold which is Einstein and either has a totally geodesic boundary or is conformally compact and has a mean convex boundary is isometric to the hyperbolic half-space. 
\end{abstract}

\maketitle

\section{Introduction}

Given a non-compact Riemannian manifold $(M,g)$ of dimension $n\geq 3$ arising as the (time-symmetric) initial data set for a solution $(\overline M,\overline g)$ of Einstein field equations in dimension $n+1$, standard physical reasoning suggests the existence of a geometric invariant defined in terms of the asymptotic behavior of the underlying metric at spatial infinity. Roughly speaking, it is assumed that in the asymptotic region $(M, g)$ converges to some reference space $(N,b)$, which by its turn is required to propagate to a {\em static} solution (i.e a solution displaying a time-like vector field whose orthogonal distribution is integrable), and the mass invariant, which is denoted by $\mathfrak m_{(g,b)}$ and should be interpreted as the total energy of the isolated gravitational system modelled by $(\overline M,\overline g)$, is designed so as to
capture the  coefficient of the leading term in the asymptotic expansion of $g$ around $b$. In particular, the important
question arises as to whether, under a suitable dominant energy condition,
the invariant in question satisfies the positive mass inequality
\begin{equation}\label{massineq}
\mathfrak m_{(g,b)}\geq  0, 
\end{equation}
with equality taking place if and only if $(M, g) = (N, b)$ isometrically.

The classical example is the asymptotically flat case, where the reference space is 
$(\mathbb R^n,\delta)$, the  Euclidean space endowed with the standard flat metric $\delta$. Here, $\mathfrak m_{(g,\delta)}$ is the so-called ADM mass and it has been conjectured that the corresponding positive mass inequality holds true whenever the scalar curvature $R_g$ of $g$ is non-negative. After previous contributions by Schoen-Yau \cite{SY1} if $n \leq 7$ and by Witten and Bartnik \cite{Wi, Bar} in the spin case, the conjecture has at last been settled in independent contributions by Schoen-Yau \cite{SY2} and Lohkamp \cite{Lo}.

Partly motivated by the
so-called
AdS/CFT correspondence, which in the Euclidean semi-classical limit highlights
Einstein metrics with negative scalar curvature,  
recently there has been much interest in studying similar
invariants for non-compact Riemannian manifolds whose geometry
at infinity approaches some reference space with constant negative sectional curvature \cite{An,He1}. A
notable example occurs in case the model geometry is hyperbolic space $(\mathbb H^n,b)$. The novelty here is that the asymptotic invariant is {\em not} a number but instead
a linear functional on 
the space $\mathcal N_b$ of static potentials $V:\mathbb H^n\to\mathbb R$ satisfying 
\begin{equation}\label{stateq}
\nabla^2_bV=Vb;
\end{equation}
see \cite{CH,CN,Mi,He1}.
However, symmetry considerations allow us to extract a mass-like invariant
(i.e. a real number) out of the given functional, so it makes sense to ask whether the inequality
similar to (\ref{massineq}) holds, with the corresponding rigidity statement characterizing the reference space. 
After preliminary contributions by Min-Oo \cite{M-O}, Anderson-Dahl \cite{AD} and Wang \cite{Wa}, the conjecture has been confirmed in case the underlying manifold is spin by Chru\'sciel-Herzlich
\cite{CH}. Elementary proofs of this result in special cases are available in \cite{DGS,dLG} (graphical manifolds) and \cite{BCN} (small perturbations of the standard hyperbolic metric). We  also refer to \cite{Ma} for a treatment of the non-time-symmetric case. Regarding the not necessarily spin case, we should mention the 
work by Andersson-Cai-Galloway \cite{ACG}.

At least in the asymptotically flat case, the positive mass inequality (\ref{massineq}) has applications that transcend its physical motivation. In particular, it has been crucially used in Schoen's solution of the Yamabe problem \cite{Sc,LP} and in the  investigation of multiplicity and compactness issues for solutions to this same problem \cite{BM, dLPZ}.    
The need to examine these questions for compact manifolds with boundary suggested the consideration of a mass-type invariant for asymptotically flat manifold with a {\em non-compact} boundary $\Sigma$ modelled on the Euclidean half-space $\mathbb R^n_+$. In \cite{ABdL} a positive mass inequality has been established for this invariant under the assumptions that the scalar curvature $R_g$ and the mean curvature $H_g$ along the boundary are both nonnegative and that the double of the underlying manifold satisfies the standard (i.e. boundaryless) mass inequality. In view of the recent progress due to Schoen-Yau and Lohkamp mentioned above, the positive mass theorem in \cite{ABdL} actually  holds in full generality.  

The purpose of this article is to devise a mass-type invariant which   
at the same time extends
those considered in \cite{ABdL} and \cite{CH}. More precisely, here we take as reference space  the hyperbolic half-space $(\mathbb H^n_+,b)$, which is obtained by cutting the standard hyperbolic space $(\mathbb H^n,b)$ along a totally geodesic hypersurface $\partial \mathbb H^n_+$. We make use of Witten's spinorial approach  to establish, for asymptotically hyperbolic spin manifolds $(M,g)$ with a non-compact boundary $\Sigma$ and which are  modeled at infinity on $(\mathbb H^n_+,b,\partial\mathbb H^n_+)$, a sharp positive mass inequality   under suitable lower bounds on the scalar curvature $R_g$ and the mean curvature $H_g$ of $\Sigma$;  see Definition \ref{def:as:hyp} and Theorems \ref{maintheo} and \ref{conjmptheo} below. The following rigidity statements are then  consequences of our main results. 

\begin{theorem}\label{maintheocor}
	Let $(M,g,\Sigma)$ be an asymptotically hyperbolic spin manifold with $R_g\geq -n(n-1)$ and $H_g\geq 0$. Assume further that $g$ agrees with the reference hyperbolic metric $b$ in a neighborhood of infinity. Then $(M,g,\Sigma)=(\mathbb H^n_+,b,\partial\mathbb H^n_+)$ isometrically. 
	\end{theorem}

\begin{theorem}\label{riggen}
	Let $(M,g,\Sigma)$ is an asymptotically hyperbolic spin manifold. Assume further that $g$ is Einstein 
	and that $\Sigma$ is totally geodesic. Then $(M,g,\Sigma)=(\mathbb H^n_+,b,\partial\mathbb H^n_+)$ isometrically. 
\end{theorem}

The proof of Theorem \ref{riggen} also makes use of an alternate definition of the mass in terms of the Einstein tensor on the interior and the Newton tensor of the boundary as described in Theorem \ref{asymhypl} below. This has recently beed  established in \cite{dLGM}  by adapting an argument first put forward by Herzlich \cite{He2} in the boundaryless case. It should be pointed out that the boundaryless version of Theorem \ref{riggen} may be obtained by a similar argument. Here we use the boundaryless version of the alternate definition (\ref{asymhypl2}), which already appears in \cite{He2}, and the positive mass theorem in \cite{CH}. This actually provides a noteworthy extension of a celebrated rigidity result by Anderson and Dahl \cite{AD} originally proved in the more restrictive setting of conformal compactness; see also \cite[Theorem 4.5]{He1}. More precisely, the following result, which may be thought of as a corollary of our proof of Theorem \ref{riggen}, holds.

\begin{theorem}\label{riggennbd}
	Let $(M,g)$ be a complete, boundaryless  spin manifold which is asymptotically hyperbolic (in the sense of \cite{CH}) and assume further that $g$ is Einstein. Then $(M,g)=(\mathbb H^n,b)$ isometrically. 
	\end{theorem}   

We now discuss another rigidity result stemming from our main theorems which is more directly related to the AdS/CFT correspondence mentioned above. Compared with Theorem \ref{riggen}, it allows us to substantially relax the assumption on the geometry of the non-compact boundary at the expense of requiring a more restrictive behavior at infinity, namely, conformal compactness. 

Let $\overline M$ be a compact $n$-manifold carrying a $(n-2)$-dimensional corner which can be written as $S\cap\Sigma$, where $S$ and $\Sigma$ are  smooth hypersurfaces of $M$ such that  $\partial \overline M=S\cup\Sigma$, with $S$ being connected. Let $g$ be a Riemannian metric on $M:={\rm int}\,\overline M\cup\Sigma$.  
We say that $(M,g)$ is {\em conformally compact} if there exists a collar neighborhood $\mathcal U\subset \overline M$ of $S$ such that on ${\rm int}\,\,\mathcal U$ we may write $g=s^{-2}\overline g$ with $\overline g$ extending to a sufficiently regular metric on $\mathcal U$ so that $S$ and $\Sigma$ meet orthogonally (with respect to $\overline g$) along their common boundary $S\cap \Sigma$, where $s:\mathcal U\to \mathbb R$ is a {\em defining function} for $S$ in the sense that  $s\geq 0$, $s^{-1}(0)=S$, $ds|_S\neq 0$ and $\nabla_{\overline g}s$ is tangent to $\Sigma$ along $\mathcal U\cap\Sigma$. 
The restriction $\overline g|_S$ defines a metric which changes by a conformal factor if the defining function is changed.
Thus, the conformal class $[\overline g|_S]$ of $\overline g|_S$ is well defined. We then say that the pair $(S,[\overline g|_S])$ is the {\em conformal infinity} of $(M,g)$,

If $|ds|_{\overline g}=1$ along $S$ then 
$(M,g)$ is {\em weakly} asymptotically hyperbolic in the sense that its 
sectional curvature converges to $-1$ as one approaches $S$. In this case, if $h_0$ is a metric on $S$ representing  the given conformal infinity then there exists a unique defining function $t$ in $\mathcal U$ so that
\begin{equation}\label{expan}
g=\sinh^{-2}t\left(dt^2+h_{t}\right),
\end{equation}
where $h_{t}$ is a $t$-dependent family of metrics on $S$ with $h_{t}|_{t=0}=h_0$. 

\begin{remark}\label{remians}{\rm 
Recall that (\ref{expan}) is established by means of a conformal deformation of the type $\widetilde g=e^{2f}\overline g$. Thus, if $\widetilde s=e^fs$ then $t$  defined by $\widetilde s=\sinh t$  is required to be the distance function to $S$ with respect to $\widetilde g$. The condition $|dt|_{\widetilde g}=1$ turns out to be a first order PDE for $f$, namely, 
\begin{equation}\label{nonchar}
\partial_sf=\frac{s}{2}\left(e^{2f}-|\nabla_{\overline g}f|^2_{\overline g}\right)+\frac{1-|\nabla_{\overline g}s|_{\overline g}^2}{2s},
\end{equation}
for which $S$ is a non-characteristic hypersurface.
Thus, (\ref{nonchar})
can be solved in $\mathcal U$ with initial data $f=0$ on $S$, so that (\ref{expan}) is retrieved by setting $\widetilde g=dt^2+h_t$, where $h_t$ is the restriction of $\widetilde g$ to the level hypersurfaces of $t$; see  \cite{AD, MP}  for further details.
If  $\xi$ is a normal unit vector field with respect to $\overline g$ along $\mathcal U\cap \Sigma$ then using that $\partial s/\partial\xi=\langle\nabla_{\overline g}s,\xi\rangle_{\overline g}=0$ we easily check from (\ref{nonchar}) that $p=\partial f/\partial \xi$ satisfies
\[
\partial_sp=s\left(e^{2f}p-(\nabla_{\overline g}f)(p)\right),
\]
and since $p=0$ for $s=0$ we see that  
$\partial f/\partial\xi=0$ along $\mathcal U\cap\Sigma$. Since 
\[
\nabla_{\widetilde g}\widetilde s=e^{-f}(s\nabla_{\overline g}f+\nabla_{\overline g}s),
\] 
this means that $\nabla_{\widetilde g}t=\cosh^{-1} t\nabla_{\widetilde g}\widetilde s$ remains tangent to $\Sigma$ along $\mathcal U\cap\Sigma$.  
}
\end{remark}

\begin{definition}\label{defasym2} 
	Let 
	$(M,g)$ be a weakly asymptotically hyperbolic manifold as above. We say that $(M,g)$ is {\em asymptotically hyperbolic} (in the conformally compact sense and with a non-compact boundary $\Sigma$) if its conformal infinity is $(\mathbb S_{+}^{n-1},[h_0])$, where $h_0$ is a round metric on $\mathbb S_{+}^{n-1}$, the unit upper $(n-1)$-hemisphere,
	and the following asymptotic expansion holds as $t\to 0$:
	\begin{equation}\label{asymexpcoll}
	h_{t}=h_0+\frac{t^n}{n!}h+k,
	\end{equation}
	where $h$ and $k$ are symmetric $2$-tensors on $\mathbb S^{n-1}_{+}$ and the remainder term $k$
	satisfies 
	\begin{equation}
	\label{remainder}
	|k|+|\nabla_{h_0} k|+|\nabla_{h_0}^2k|=o(t^{n+1}).
	\end{equation}
\end{definition} 

It turns out that a manifold which is asymptotically hyperbolic in this sense is also asymptotically hyperbolic in the sense of Definition \ref{def:as:hyp} below, so we may assign to it a mass-type invariant which extends Wang's construction in the boundaryless case \cite{Wa}. Thus, the rigidity statement of  Theorem \ref{conjmptheo} yields another natural extension to our setting of  the result by Andersson-Dahl \cite{AD,He1} mentioned above in connection with Theorems \ref{riggen} and \ref{riggennbd}.

\begin{theorem}\label{rigconfcom}
	Let $(M,g,\Sigma)$ be a conformally compact, asymptotically hyperbolic spin manifold as above. Assume further that $g$ is Einstein 
	and that the mean curvature of $\Sigma$ is everywhere nonnegative. Then $(M,g,\Sigma)=(\mathbb H^n_+,b,\partial \mathbb H^n_+)$ isometrically. 
	\end{theorem}

 \begin{remark}\label{remove}
 	{\rm In recent years there have been considerable efforts by several authors in the direction of  removing the spin assumption in the seminal rigidity result appearing in \cite{AD}, the difficulty here coming from the fact that no suitable positive mass theorem has been proved in this generality; see for instance \cite{CLW,LQS,Ra} and the references therein. In any case, one might ask whether analogous developments hold in the presence of a boundary as in Theorem \ref{rigconfcom} above. On the other hand, Theorems \ref{riggen} and \ref{riggennbd} show that conformal compactness is not really needed in order to obtain a rigidity statement as long as we remain in the spin category, where the pertinent positive mass inequality is available. This suggests a  more ambitious goal, namely, to investigate whether the corresponding rigidity persists for Einstein metrics (and totally geodesic boundaries) in the general asymptotically hyperbolic setup of Definition \ref{def:as:hyp}.}
 	\end{remark}
	
This paper is organized as follows. In Section \ref{prelim} we define the relevant class of asymptotically hyperbolic manifolds with a non-compact boundary and in Section \ref{geomass} we attach to each manifold in this class a mass functional whose geometric invariance is established. The proofs of the positive mass inequality for spin manifolds and its geometric consequences, including the rigidity statements above, are presented in Section \ref{mainsec}. This uses some preparatory material regarding spinors on manifolds with boundary which is discussed in Section \ref{spinorsbd}.

\section{Asymptotically hyperbolic manifolds}\label{prelim}

Recall that the {hyperboloid model} for hyperbolic space in dimension $n$ is given by 
$$
\mathbb H^n=\{x\in\R^{1,n}; x_0>0, \langle x,x\rangle_L=-1\},
$$ 
where $\R^{1,n}$ is the Minkowski space with the standard flat metric 
\[
\langle x,x\rangle_L=-x_0^2+x_1^2+\cdots+x_n^2.
\] 
The reference space we are interested in is $(\mathbb H^n_+,b,\partial\mathbb H^n_+)$, 
where $\mathbb H_+^n=\{x\in\mathbb H^n;x_n\geq 0\}$ is the {\em hyperbolic half-space} endowed with
the induced metric 
$$
b=\frac{dr^2}{1+r^2}+r^2h_0,
$$
where $h_0$ stands for the canonical metric on the unit hemisphere $\mathbb S^{n-1}_+$, and 
\[
r=\sqrt{x_1^2+\cdots +x_n^2}.
\] 
Note that $\mathbb H^n_+$
carries a non-compact, totally geodesic boundary, namely, $\partial \mathbb H^n_+=\{x\in\mathbb H^n_+;x_n=0\}$.
Recall that the space of static potentials $\mathcal N_b$ on $\mathbb H^n$ is spanned by   
$V_{(0)},V_{(1)},\cdots,V_{(n)}$, where $V_{(i)}=x_i|_{\mathbb H^n}$. It is easy to check that 
	\begin{equation}\label{statpotbd}
		\frac{\partial V_{(i)}}{\partial\eta}=0, \quad i\neq n,
	\end{equation}
	where $\eta$ is the outward unit normal to $\partial\mathbb H^n_+$.
Thus, we are led to consider 
\[
\mathcal N_b^+=\left\{V\in \mathcal N_b;\frac{\partial V}{\partial\eta}=0\right\}, 
\]
which is spanned by $V_{(0)},V_{(1)},\cdots,V_{(n-1)}$. In particular, $V=O(r)$ as $r\to +\infty$ for any $V\in\mathcal N_b^+$. 

\begin{remark}\label{ballmodel}{\rm                                                     
The construction above can alternatively be carried out in the context of the so-called { Poincar\'e model} of hyperbolic geometry. Hence, we may consider the half $n$-disk $\mathbb B^n_+=\{x'=(x_1,\cdots,x_n)\in\mathbb R^n\:|\:|x'|\leq 1, x_n\geq 0\}$ endowed with the metric
\begin{equation}\label{confrep}
\widehat b=\omega(x')^{-2} \delta,\quad \omega(x')=\frac{1-|x'|^2}{2},
\end{equation}
where $\delta$ is the standard Euclidean metric.
The space $\mathcal N_{\widehat b}^+$ now is spanned by the functions 
$$
\widehat V_{(0)}(x')=\frac{1+|x'|^2}{1-|x'|^2}, \:\widehat V_{(1)}(x')=\frac{2x_1}{1-|x'|^2},\: \cdots\:, \:\widehat V_{(n-1)}(x')=\frac{2x_{n-1}}{1-|x'|^2}. 
$$
Under the isometry  between $(\mathbb H^n_+,b)$ and $(\mathbb B^n_+,\widehat b)$ given by stereographic projection ``centered'' at $(-1,0,\cdots,0)$, 
$\partial\mathbb  H^n_+$ is mapped onto the unit $(n-1)$-disk $\partial \mathbb B_+^n$ defined by $x_n=0$.  Whenever convenience demands we will interchange freely between these models without further notice. 
}
\end{remark}

\begin{remark}\label{staticsp}{\rm 
		The linear space $\mathcal N_b^+$ can actually be thought of as a space of static potentials on $\mathbb H^n_+$ as follows. 
		It is well known that a vacuum solution $(\overline M^{n+1},\overline g)$ of the Einstein field equations can be characterized as an extremizer  of the Einstein-Hilbert functional 
		$$
		\overline g\mapsto\int_{\overline M} (R_{\overline g}-2\Lambda)dv_{\overline g},
		$$
		with $\Lambda\in \mathbb R$. In the case $\d {\overline M}\neq \emptyset$, it is natural to consider instead the Gibbons-Hawking-York action
		$$
		\overline g\mapsto \int_{\overline M} (R_{\overline g}-2\Lambda)\,dv_{\overline g}+\int_{\d {\overline M}} 2(H_{\overline g}-\lambda)\,d\sigma_{\overline g}
		$$
		whose critical metrics satisfy the system
		\begin{equation}\label{eq:spacetime}
		\begin{cases}
		{\rm Ric}_{\overline g}-\frac{1}{2}R_{\overline g}\overline{g}+\Lambda\overline g=0, &\text{in}\:{\overline M},
		\\
		\Pi_{\overline g}-H_{\overline g} \overline g|_{\d\overline M}+\lambda \overline g|_{\d\overline M}=0,&\text{on}\:\d {\overline M},
		\end{cases}
		\end{equation}
		where $\Pi$ is the boundary second fundamental form and $\lambda\in\mathbb R$.
		The first equation of (\ref{eq:spacetime}) is the Einstein field equation for a vacuum spacetime with cosmological constant $\Lambda$, while the second one introduces the constant $\lambda$ which is related to the geometry of $\partial\overline M$. Recall that the spacetime $(\overline M,\overline g)$ is said to be {\em static} if it carries a time-like Killing vector field $\xi$ whose orthogonal distribution is integrable or, equivalently, $\overline g$ can be written as a warped product ${\overline g}=-V^2dt^2+g$, where $g$ is a Riemannian metric on the spacelike slice $M^n$ and $V^2=-\overline g(\xi,\xi)$.
The leaves of this foliation are totally geodesic and isometric to each other. If  $\d M\neq \emptyset$, those leaves are orthogonal to $\partial\overline M$ which corresponds to $\xi$ being tangent to $\partial\overline M$. In this case, the system (\ref{eq:spacetime}) is equivalent to requiring that $g$ and $V$ satisfy
		\begin{equation}\label{eq:static:1}
		\begin{cases}
		-V {\rm Ric}_g+\nabla_g^2V+\widetilde\Lambda Vg=0,&\text{in}\:M,
		\\
		\Pi_g-\widetilde\lambda g|_{\partial M}=0,&\text{on}\:\d M,
		\\
		\Delta_g V=-\widetilde\Lambda V, &\text{in}\:M,
		\\
		\displaystyle\frac{\d V}{\d\eta}=\widetilde\lambda V, &\text{on}\:\d M,
		\end{cases}
		\end{equation}
		where $\widetilde \Lambda=\frac{2}{n-1}\Lambda$ and $\widetilde\lambda=\frac{1}{n-1}\lambda$. Thus, by setting $\widetilde\Lambda=-n$ and $\widetilde\lambda=0$,
		the elements $V\in\mathcal N_b^+$ are solutions of  (\ref{eq:static:1}) when $(M,g)=(\mathbb H_+^n, b)$.
	}
\end{remark}

We now define the  notion of an asymptotically hyperbolic manifold with a non-compact boundary having $(\mathbb H^n_+,b,\partial\mathbb H^n_+)$ as a model; this should be compared with the related boundaryless concept in \cite{CH}.
Recall that $\mathbb H^n_+$ can be parameterized by polar coordinates $(r,\theta)$, where $r=\sqrt{x_1^2+\cdots+x_{n}^2}$ and $\theta=(\theta_2,\cdots,\theta_{n})\in \mathbb S^{n-1}_+$. 
For all $r_0>0$ large enough let us set $\mathbb H^n_{+,r_0}=\{x\in\mathbb H_+^n;r(x)\geq r_0\}$.

\begin{definition}\label{def:as:hyp}
	We say that $(M^n,g,\Sigma)$ is {\it{asymptotically hyperbolic}} (with a non-compact boundary $\Sigma$) if there exist $r_0>0$, a region $M_{{\rm ext}}\subset M$ and a diffeomorphism
	$
	F:\mathbb H_{+,r_0}^n\to  M_{{\rm ext}}
	$
	such that: 
	\begin{enumerate}
	 \item As $r\to+\infty$,  $e:=F^*g-b$ satisfies 
	\begin{equation}\label{asympthyp}
	|e|_b+|\nabla_b e|_b+|\nabla^2_be|_b=O(r^{-\tau}), \quad \tau>\frac{n}{2};
	\end{equation}
	\item both $\int_M r(R_g+n(n-1))dM$ and $\int_\Sigma rH_gd\Sigma$ are finite, where the asymptotical radial coordinate $r$ has been smoothly extended to $M$.
	\end{enumerate}
\end{definition} 
\begin{remark}\label{rmk:bd}
{\rm Although $\Sigma$ may be disconnected, it follows from Definition \ref{def:as:hyp} that $\Sigma$ has exactly one non-compact component.}
\end{remark}

\section{The mass functional and its geometric invariance}\label{geomass}

Here we define the mass functional for an asymptotically hyperbolic manifold and establish its geometric invariance.  Given a chart at infinity $F: \mathbb H^n_{+,r_0}\to M_{\rm ext}$ as in Definition \ref{def:as:hyp} we set, for $r_0<r<r'$,
$
A_{r,r'}=\{x\in \mathbb H^n_{+,r_0};r\leq |x| \leq r'\}$,
$\Sigma_{r,r'}=\{x\in \partial \mathbb H_{+,r_0}^n;r\leq |x| \leq r'\}
$ and
$S^{n-1}_{r,+}=\{x\in \mathbb H_{+,r_0}^n|;|x|=r\}$, 
so that 
$$\d A_{r,r'}=S^{n-1}_{r,+}\cup \Sigma_{r,r'}\cup S^{n-1}_{r',+}.$$
We represent by $\mu$ the outward unit normal vector field to $S^{n-1}_{r,+}$ or $S^{n-1}_{r',+}$, computed with respect to the reference metric $b$. Also, we consider $S^{n-2}_r=\partial S_{r,+}^{n-1}\subset \partial \mathbb H^n_{+,r}$, endowed with its outward unit conormal field $\vartheta$, again  with respect to $b$.
We set $e=g-b$, where we have written $g=F^*g$ for simplicity of notation, and we define the $1$-form
\begin{equation}\label{charge}
\mathbb U(V,e)=V({\rm div}_be-d{\rm tr}_be)-{\nabla_bV}\righthalfcup e+{\rm tr}_be\, dV,
\end{equation}
for a static potential $V\in\mathcal N_b^+$. 

\begin{theorem}\label{finitemass}
	If $(M,g,\Sigma)$ is an asymptotically hyperbolic manifold then the quantity
	\begin{equation}\label{massdef}
	\mathfrak m_{(g,b,F)}(V)=\lim_{r\to +\infty}\left[\int_{S^{n-1}_{r,+}}\langle\mathbb U(V,e),\mu\rangle dS^{n-1}_{r,+}-
	\int_{S^{n-2}_{r}}Ve(\eta,\vartheta)dS^{n-2}_{r}\right]
	\end{equation}
exists and is finite.	
	\end{theorem}

\begin{proof}
The argument is based on the Taylor expansion, as $r\to +\infty$, of the scalar curvature $R_g$ around the reference metric $b$; see \cite{Mi, He1} for nice descriptions of the underlying strategy. We have 
\[
R_g=-n(n-1)+\dot R_be+\rho_b(e), 
\]
where 
\[
\dot R_be={\rm div}_b({\rm div}_be-d{\rm tr}_be)+(n-1){\rm tr}_be
\]
is the linearization of the scalar curvature at $b$  
and $\rho_b(e)=O(r^{-2\tau})$ denotes terms which are at least quadratic in $e=g-b$.
The key point is to properly handle the linear term. 
Indeed, a well-known computation gives 
the fundamental identity
\[
V\dot R_be=  \langle \dot R_b^*V,e\rangle +{\rm div}_b\mathbb U(V,e), 
\]
where 
\[
\dot R_b^*V=\nabla^2_bV-Vb
\]
is the formal  $L^2$ adjoint of $\dot R_b$. Since $V\in \mathcal N_b^+$ we thus get from (\ref{stateq}),
\begin{equation}\label{key2}
V(R_g+n(n-1))={\rm div}_b\mathbb U(V,e)+\rho_b(V,e), 
\end{equation}
where $\rho_b(V,e)=V\rho_b(e)=O(r^{-2\tau+1})$ since $V=O(r)$.

We now perform the integration of (\ref{key2}) over the half-annular region $A:=A_{r,r'}\subset \mathbb H^n_+$ and eventually explore the imposed boundary conditions, namely, 
\begin{equation}\label{Vbd}
\Pi_b=0,\quad \frac{\partial V}{\partial\eta}=0,
\end{equation}
where $\Pi_b$ is the second fundamental form of $\partial\mathbb H^n_+$ and $V\in\mathcal N_b^+$; see (\ref{statpotbd}).
Integration yields 
\begin{eqnarray*}
	\int_{A}V(R_g-R_b)dA_b 
	& = & \int_{S^{n-1}_{r',+}}\langle\mathbb U(V,e),\mu\rangle dS^{n-1}_{r,+}-
	\int_{S^{n-1}_{r,+}}\langle\mathbb U(V,e),\mu\rangle dS^{n-1}_{r,+}\\
	&  & +\int_{\Sigma_{r,r'}}V({\rm div}_be-d{\rm tr}_be)(\eta)d\Sigma_\beta\\
	& & -\int_{\Sigma_{r,r'}}({\nabla_bV}\righthalfcup e)(\eta)d\Sigma_\beta
	+ \int_{A}\rho_b(V,e)dA_b,  
\end{eqnarray*}
where $\beta=b|_{\Sigma}$ and we used that $dV(\eta)=0$ by (\ref{Vbd}).
We now observe that 
the mean curvature varies as 
\[
2\dot H_be=[d{\rm tr}_be-{\rm div}_be](\eta)-{\rm div}_\beta X-\langle \Pi_b,e\rangle_\beta,
\]
where $X$ is the vector field dual to the $1$-form $e(\eta,\cdot)|_{T\Sigma}$ and we still denote $e=e|_\Sigma$. 
Hence, upon expansion around the reference metric $\beta$ and taking into account that $\Pi_b=0$ by (\ref{Vbd}),
\begin{eqnarray*}
	2\int_{\Sigma_{r,r'}}VH_gd\Sigma_{\beta}
	& = & \int_{\Sigma_{r,r'}}V(d{\rm tr}_be-{\rm div}_be)(\eta)d\Sigma_\beta-\int_{\Sigma_{r,r'}}V{\rm div}_{\beta}Xd\Sigma_\beta\\
	& & \quad +\int_{\Sigma_{r,r'}}\rho_\beta(V,e)d\Sigma_{\beta},
\end{eqnarray*}
where $\rho_\beta(V,e)=O(r^{-2\tau+1})$. Thus, if 
\[
\mathcal A_{r,r'}(g):=\int_{A}V(R_g+n(n-1))dA_b+2\int_{\Sigma_{r,r'}}VH_gd\Sigma_{\beta},
\]
then
\begin{eqnarray*}
	\mathcal A_{r,r'}(g)
	& =  & 	\int_{S^{n-1}_{r',+}}\langle\mathbb U(V,e),\mu\rangle dS^{n-1}_{r,+}-
	\int_{S^{n-1}_{r,+}}\langle\mathbb U(V,e),\mu\rangle dS^{n-1}_{r,+}\\
	& & -\int_{\Sigma_{r,r'}}({\nabla_bV}\righthalfcup e)(\eta)d\Sigma_\beta-\int_{\Sigma_{r,r'}}V{\rm div}_{\beta}Xd\Sigma_\beta\\
	& & +\int_{A}\rho_b(V,e)dA_b+ \int_{\Sigma_{r,r'}}\rho_\beta(V,e)d\Sigma_{\beta}.
\end{eqnarray*}
But notice that 
\begin{eqnarray*}
	-V{\rm div}_\beta X 
	& = & -{\rm div}_\beta(VX)+\langle\nabla_bV,X\rangle_\beta\\
	& = & -{\rm div}_\beta(VX)+e(\nabla_bV,\eta)\\
	& = & -{\rm div}_\beta(VX)+({\nabla_bV}\righthalfcup e)(\eta),
\end{eqnarray*}
and we end up with 
\begin{eqnarray*}
	\mathcal A_{r,r'}(g)
	& =  & 	\int_{S^{n-1}_{r',+}}\langle\mathbb U(V,e),\mu\rangle dS^{n-1}_{r,+}-
	\int_{S^{n-1}_{r,+}}\langle\mathbb U(V,e),\mu\rangle dS^{n-1}_{r,+}\\
	& & -\int_{\Sigma_{r,r'}}{\rm div}_\beta(VX)d\Sigma_{\beta}\\
	& & +\int_{A}\rho_b(V,e)dA_b+ \int_{\Sigma_{r,r'}}\mathcal \rho_\beta(V,e)d\Sigma_{\beta}\\
	& =  & 	\left(\int_{S^{n-1}_{r',+}}\langle\mathbb U(V,e),\mu\rangle dS^{n-1}_{r,+}-
	\int_{S^{n-2}_{r'}}Ve(\eta,\vartheta) dS^{n-2}_{r}\right)\\
	& & -\left(\int_{S^{n-1}_{r,+}}\langle\mathbb U(V,e),\mu\rangle dS^{n-1}_{r,+}-
	\int_{S^{n-2}_{r}}Ve(\eta,\vartheta) dS^{n-2}_{r}\right)\\
	& & +\int_{A_{r,r'}}\rho_b(V,e)dA_b+ \int_{\Sigma_{r,r'}}\rho_\beta(V,e)d\Sigma_{\beta}.
\end{eqnarray*}
Now, 
the assumption $\tau>n/2$ in (\ref{asympthyp}) implies that the last two integrals vanish as $r\to +\infty$. Also, since $V=O(r)$ the integrability assumptions in Definiton \ref{def:as:hyp}, (2), imply that $\mathcal A_{r,r'}(g)\to 0$ as $r\to +\infty$ and the result follows.  
\end{proof}

We should think of (\ref{massdef}) as defining a linear functional
\[
\mathfrak m_{(g,b,F)}:\mathcal N_b^+\to\mathbb R.
\]
We note however that the decomposition $g=b+e$ used above depends on the choice of an admissible chart at infinity (the diffeomorphism $F$), so  we need to check that $\mathfrak m_{(g,b,F)}$ behaves properly as we pass from one such chart to another. For this we need some preliminary results. 

\begin{lemma}\label{exact}
	If $V\in\mathcal N_b^+$ and $X$ is a vector field then 
	\begin{equation}\label{exact1}
	\mathbb U(V,\mathcal L_Xb)={\rm div}_b\mathbb V(V,X,b),
	\end{equation}
	with the $2$-form $\mathbb V$ being explicitly given by 
	\begin{equation}\label{exact2}
	\mathbb V_{ik}=V(X_{i;k}-X_{k;i})+2(X_kV_i-X_iV_k),
	\end{equation}
	where the semi-colon denotes covariant derivation with respect to $b$.
	\end{lemma}

\begin{proof}
Using $(\mathcal L_Xb)_{ij}=X_{i;j}+X_{j;i}$ and (\ref{charge}) we see that 
\[
\mathbb U_i=I_i^{(1)}+I_i^{(2)}+I_i^{(3)},
\]
where 
\begin{eqnarray*}
I_i^{(1)} &= & V(X_{i;k}^{\:\:\:k}+X^k_{\:\:;ik}-2X^k_{\:\:;ki})\\
& = & V(X_{i;k}^{\:\:\:k}+X^k_{\:\:;ik}-2X^k_{\:\:;ik}+2{\rm Ric}^b_{ik}X^k)
\\
&  =  & V(X_{i;}^{\:\:\:k}-X^k_{\:\:;i})_{;k}-2(n-1)VX_i,
\end{eqnarray*}
\begin{eqnarray*}
I_i^{(2)} &= & -(X_{i;k}+X_{k;i})V^k\\
& = & (X_{i;k}-X_{k;i})V^k-2X_{i;k}V^k\\
&= & (X_{i;k}-X_{k;i})V^k-2(X_{i}V^k)_{;k}+2X_i\Delta_b V,
\end{eqnarray*}
and 
\begin{eqnarray*}
I_i^{(3)}&=& 2X^k_{\:\:\:;k}V_ i\\
& = & 2(X^kV_i)_{;k}-2X^k(\nabla^2_bV)_{ik}.
\end{eqnarray*}
Thus, 
\begin{eqnarray*}
\mathbb U_i 
& = &
\left(V(X_{i;}^{\:\:k}-X^k_{\:\:;i})-2X_iV^k+2X^kV_ i\right)_{;k}\\
& & \quad +2(-(n-1)VX_i-(\nabla^2_bV)_{ik}X^k+\Delta_bVX_i)\\
& = &
\left(V(X_{i;}^{\:\:k}-X^k_{\:\:;i})-2X_iV^k+2X^kV_ i\right)_{;k}\\
& & \quad +2(Vb-\nabla^2_bV)_{ik}X^k,
\end{eqnarray*}
and the last term drops out since $V\in\mathcal N_b^+$. Hence,  
$\mathbb U_i=b^{jk}\mathbb V_{ij;k}$.
\end{proof}

The next result shows that the reference space $(\mathbb H^n_+,b,\partial\mathbb H^N_+)$ is rigid at infinity in the appropriate sense.

\begin{lemma}\label{rigid}
	If $F:\mathbb H^n_+\to \mathbb H^n_+$ is a diffeomorphism such that $F^*b=b+O(r^{-\tau})$ as $r\to +\infty$ then there exists an isometry $\mathcal I$ of $(\mathbb H^n_+,b)$ which preserves $\partial\mathbb H^n_+$ and satisfies 
	\[
	F=\mathcal I+O(r^{-\tau}),
	\] 
with a similar estimate holding for the first order derivatives.
	\end{lemma}

\begin{proof}
	This is established by straightforwardly adapting the reasoning in the proofs of the corresponding results in 
	\cite{CH,CN}. Therefore, the argument is omitted.
\end{proof}

Suppose now that we have two diffeomorphisms, say $F_{1},F_2:\mathbb H^n_{+,r_0}\to M_{\rm ext}$, definining charts at infinity as above and consider $F=F^{-1}_1\circ F_2:\mathbb H^n_{+,r_0}\to \mathbb H^n_{+,r_0}$. It is clear that $F^*b=b+O(r^{-\tau})$, $\tau>n/2$, so by the previous lemma, $F=\mathcal I+O(r^{-\tau})$ for some isometry $\mathcal I$ preserving $\partial \mathbb H^n_+$. The next result establishes the geometric invariance of the mass functional appearing in Theorem \ref{finitemass}.

\begin{theorem}\label{geoinv}
	Under the conditions above, 
	\begin{equation}\label{geoinv2}
	\mathfrak m_{(g,b,F_1)}(V\circ\mathcal I)=\mathfrak m_{(g,b,F_2)}(V), \quad V\in\mathcal N_b^+.
	\end{equation}
	\end{theorem} 

\begin{proof}
Again we proceed as in \cite{Mi}. 
We may assume that $\mathcal I$ is the identity, so that $F={\rm exp}\circ \zeta$, where $\zeta=O(r^{-\tau})$ is a vector field everywhere tangent to $\Sigma$. Also, since the mass is defined  asymptotically, we may assume that $M$ is diffeomorphic to $\mathbb H^n_+$, so that $M_{\rm ext}$ can be identified to $\mathbb H_{+,r_0}^n$. For $r>r_0$, let $\mathcal S^{n-1}_r=\{x\in\partial\mathbb H^n_{+};|x|\leq r\}$, so that $\mathcal S^{n-1}_{r}\cup S_{r,+}^{n-1}$ is the boundary of the compact region on $\mathbb H^n_+$ defined by $|x|\leq r$. Now set
\[
e_1=g_1-b,\quad g_1=F_1^*g,
\]
and
\[
e_2=F^*_2g-b=F^*g_1-b,
\]
so that 
\[
E:=e_2-e_1=F^*g_1-g_1=F^*(b+e_1)-(b+e_1)=\Lc_\zeta b+R_1,
\]
where $R_1=O(r^{-2\tau})$ is a remainder term. It follows that 
\[
\mathbb U(V,E)=\mathbb U(V,\Lc_\zeta b)+R_2,
\]
where $R_2=\mathbb U(V,R_1)=O(r^{-2\tau+1})$ vanishes after integration over $S^{n-1}_{r,+}$ as $r\to +\infty$. It follows that 
\begin{eqnarray*}
\lim_{r\to+\infty}\int_{S^{n-1}_{r,+}}\langle\mathbb{U}(V,E),\mu\rangle dS^{n-1}_{r,+} 
& = &\lim_{r\to+\infty}\int_{S^{n-1}_{r,+}}\langle\mathbb{U}(V,\Lc_\zeta b),\mu\rangle dS^{n-1}_{r,+}\\
& = & 
	-\lim_{r+\infty}\int_{\mathcal S^{n-1}_r}\langle\mathbb{U}(V,\Lc_\zeta b),\eta\rangle d\Sigma_s,
\end{eqnarray*}
where in the last step we used (\ref{exact1}) to transfer the integral to $\mathcal S^{n-1}_r$.
We shall compute the limit above  by means of (\ref{exact2}).
Using an orthonormal  $b$-frame $\{\mathfrak f_i\}_{i=1}^n$ so that   $\eta=-\mathfrak f_n=(0, 0,\cdots,-1)$ along the boundary, 
\[
\langle\mathbb{U}(V,\Lc_\zeta b),\eta\rangle=b^{jk}\mathbb V_{ij;k}\eta^i=-\mathbb V_{nk;k}=-\mathbb V_{n\alpha;\alpha}, \quad ({\rm sum}\,\,{\rm over}\,\,1\leq\alpha\leq n-1),
\]
which shows that 
\[
\langle\mathbb{U}(V,\Lc_\zeta b),\eta\rangle|_{\Sigma_r}={\rm div}_\beta (\eta\righthalfcup\mathbb V).
\]
Thus,
\[
\int_{\mathcal S^{n-1}_r}\langle\mathbb{U}(V,\Lc_\zeta b),\eta\rangle d\Sigma_r=\int_{ S^{n-2}_r}\mathbb V(\eta,\vartheta)d\Sigma_s=-\int_{S^{n-2}_{r}}\mathbb V_{n\alpha}\vartheta^\alpha dS^{n-2}_s.
\]
But 
\[
\mathbb V_{n\alpha}=V(\zeta_{n;\alpha}-\zeta_{\alpha;n})+2(\zeta_\alpha V_n-\zeta_nV_\alpha)=-V\zeta_{\alpha;n},
\]
where we used that $\zeta_{n;\alpha}=0$ and $V_n=0$ (due to (\ref{Vbd})) and $\zeta_n=0$ (since $\zeta$ is tangent to $\Sigma$). It follows that 
\begin{equation}\label{lasti}
\lim_{r\to +\infty}\int_{S^{n-1}_{r,+}}\langle\mathbb{U}(V,E),\,\mu\rangle dS^{n-1}_{r,+}=\lim_{r\to +\infty}\int_{S^{n-2}_{r}}V\zeta_{\alpha;n}\vartheta^\alpha dS^{n-2}_r.
\end{equation}  
To complete the argument, notice that 
\[
\left(-\int_{S^{n-2}_{r}}Ve_2(\eta,\vartheta)dS^{n-2}_{r}\right)-\left(-\int_{S^{n-2}_{r}}Ve_1(\eta,\vartheta)dS^{n-2}_{r}\right)=-\int_{S^{n-2}_{r}}VE(\eta,\vartheta)dS^{n-2}_{r},
\]
and this equals
\[
-\int_{S^{n-2}_{r}}V(\Lc_\zeta b)(\eta,\vartheta)dS^{n-2}_{r}-\int_{S^{n-2}_{r}}VR_1(\eta,\vartheta)dS^{n-2}_{r},
\]
where the last integral vanishes at infinity. Finally, the remaining integral may be evaluated as 
\begin{eqnarray*}
	-\int_{S^{n-2}_{r}}V(\Lc_\zeta b)(\eta,\vartheta)dS^{n-2}_{r}
	& = & -\int_{S^{n-2}_{r}}V(\zeta_{\alpha;n}+\zeta_{n;\alpha})\vartheta^\alpha dS^{n-2}_{r}\\
	& = & -\int_{S^{n-2}_{r}}V\zeta_{\alpha;n}\vartheta^\alpha dS^{n-2}_r,
\end{eqnarray*}
which clearly cancels out the contribution coming from the right-hand side of (\ref{lasti}). We have thus shown that 
\[
\lim_{r\to +\infty}\left[\int_{S^{n-1}_{r,+}}\langle\mathbb{U}(V,E),\mu\rangle dS^{n-1}_{r,+}-\int_{S^{n-2}_{r}}VE(\eta,\vartheta)dS^{n-2}_{r}\right]=0,
\]
which finishes the proof.
\end{proof}

\begin{remark}
	\label{massbterm}
	{\rm Recalling that  $\eta=-\mathfrak f_n$ we have    $e(\eta,\vartheta)=-e(\mathfrak f_n,\vartheta^\alpha\mathfrak f_\alpha)=-e_{\alpha n}\vartheta^\alpha$. 	
		Thus, the mass functional in (\ref{massdef}) may be expressed in terms of $e_{ik}=e(\mathfrak f_i,\mathfrak f_k)$ as 
	\begin{eqnarray}\label{massdef2}
\mathfrak m_{(g,b,F)}(V) & = & \lim_{r\to +\infty}\int_{S^{n-1}_{r,+}}
\left(V\left(e^k_{i;k}-e^k_{k;i}\right)-e_{ik}V^k+e^k_kV_i\right)\mu^i\, dS^{n-1}_{r,+}\nonumber\\
& & \quad +\lim_{r\to+\infty}
\int_{S^{n-2}_{r}}Ve_{\alpha n}\vartheta^\alpha dS^{n-2}_{r}.
\end{eqnarray}
In particular, if we express the metrics in Fermi coordinates around the boundary we get $e_{\alpha n}=0$, so in these new coordinates the last integral in the righthand side of (\ref{massdef2}) does not contribute to the mass. We will use this choice of asymptotic coordinates and the ensuing simplified expression for the mass in the proof of Theorem  \ref{maintheo} below. It is not hard to check that under this change of coordinates  (\ref{geoinv2}) holds true with $\mathcal I$ being the identity isometry, so the mass functional remains the same.
% and we have 
%\begin{equation}\label{massterm3}
%\mathfrak m_{g,b,F}(V)  = \lim_{r\to +\infty}\int_{S^{n-1}_{r,+}}
%\left(V\left(e^k_{i,k}-e^k_{k;i}\right)-e_{ik}V^k+e^k_kV_i\right)\mu^i\, dS^{n-1}_{r,+}.
%\end{equation} 
}
\end{remark}

To properly appreciate the relevance of Theorem \ref{geoinv}, we note that the isometry group $O^+(n-1,1)$ of the reference space $(\mathbb H^n_{+},b,\partial \mathbb H^n_{+})$, which is formed by those isometries of $\mathbb H^n$ preserving $\partial\mathbb H^n_+$, acts naturally on $\mathcal N_b^+$, which is generated by $\{V_{(0)}, V_{(1)},\cdots,V_{(n-1)}\}$,  in such a way  that the Lorentzian metric
\[
\langle\langle z,w\rangle\rangle=z_0w_0-z_1w_1-\cdots- z_{n-1}w_{n-1}
\]
is preserved. Here, we regard $\{V_{(a)}\}_{a=0}^{n-1}$ as an orthonormal basis and endow $\mathcal N_b^+$   
with a time orientation by declaring that $V_{(0)}$ is future directed. Thus, if for any admissible chart at infinity $F$ we set
\[
\mathcal P^{[F]}_a=\mathfrak m_{(g,b,F)}(V_{(a)}), \quad 0\leq a\leq n-1,
\]
then Theorem \ref{geoinv}
guarantees that  $\langle\langle\mathcal P^{[F]},\mathcal P^{[F]}\rangle\rangle$, the past/fut\-ure pointing nature and the causal character of $\mathcal P^{[F]}$ are {\em chart independent} indeed. Combined with the standard physical reasoning, this suggests the following conjecture.

\begin{conjecture}\label{conjmp}(Positive mass)
	Let $(M,g,\Sigma)$ be an asymptotically hyperbolic manifold with $R_g\geq -n(n-1)$ and $H_g\geq 0$. Then for  any admissible chart $F$ the vector $\mathcal P^{[F]}$ is time-like  future directed, unless it vanishes and $(M,g,\Sigma)$ is isometric to $(\mathbb H^n_{+},b,\partial \mathbb H^n_+)$.
	\end{conjecture}

The ``positive mass'' terminology is justified by the fact that whenever the conjecture holds true we may define the numerical invariant
\[
\mathfrak m_{(g,b)}=\sqrt{\langle\langle \mathcal P^{[F]},\mathcal P^{[F]}\rangle\rangle},
\]
which happens to be independent of the chosen chart. This may be regarded as the total mass of the isolated gravitational system whose (time-symmetric) initial data set is $(M,g)$. Notice that in this case we always have $\mathfrak m_{(g,b)}>0$ unless  $(M,b,\Sigma)$ is isometric to $(\mathbb H^n_+,b,\partial \mathbb H_+^n)$. As remarked in the Introduction, our main results (Theorems \ref{maintheo} and \ref{conjmptheo} below) confirm the conjecture in case $M$ is spin.

For further reference we observe the existence of an alternate version of the asymptotic definition for the mass functional  (\ref{massdef}) in terms of the Einstein tensor 
\[
G_g={\rm Ric}_g-\frac{R_g}{2}g
\] 
in the interior and the Newton tensor
\[
J_g=\Pi_g-H_gg,
\]
where $\Pi_g$ is the second fundamental form along the boundary,
which has been recently established in \cite{dLGM}.  More precisely, in terms of the modified Einstein tensor 
\[
\widehat G_g=G_g-\frac{(n-1)(n-2)}{2}g
\]
we have the following result.

\begin{theorem}\cite[Theorem 4.2]{dLGM}\label{asymhypl} For each  $a=0,1,\cdots,n-1$ there exists a conformal vector field $X_a$ on $\mathbb H^n_+$ which is everywhere tangent to $\partial\mathbb H^n_+$ and which for any asymptotically hyperbolic manifold $(M,g,\Sigma)$ and admissble chart $F$ satisfies
	\begin{equation}\label{asymhypl2}
	\mathfrak m_{(g,b,F)}(V_{(a)})=d_n\lim_{r\to+\infty}\left[\int_{\breve S^{n-1}_{r,+}}\widehat G_g(F_*X_a,\mu_g)d{\breve S^{n-1}_{r,+}}+\int_{\breve S^{n-2}_r}J_g(F_*X_a,\vartheta_g)d{\breve S^{n-2}_r} \right],
	\end{equation}
	where $d_n>0$ is a dimensional constant, $\mu_g$ the outward unit normal to $\breve S^{n-1}_{r,+}:= F(S^{n-1}_{r,+})\subset M$ and similarly for $\vartheta_g$.
\end{theorem}

\begin{remark}\label{remeins}
	{\rm In general we may write 
		\[
		\widehat G_g=\widetilde G_g+\frac{2-n}{2n}(R_g+n(n-1))g,
		\]
		where
		\[
		\widetilde G_g={\rm Ric}_g-\frac{R_g}{n}g
		\]
		is the traceless Ricci tensor. In particular, if $g$ is Einstein with ${\rm Ric}_g=-(n-1)g$, so that $R_g=-n(n-1)$, then $\widetilde G_g=0$ and hence $\widehat G_g=0$ as well.
	}
\end{remark}

\section{Spinors on manifolds with boundary}\label{spinorsbd}

In this section we review the results in the theory of spinors on  $n$-manifolds carrying a (possibly non-compact) boundary $\Sigma$ which are needed in the rest of the paper.  Our presentation uses the setup introduced in  \cite{HM,HMR} and we refer to these works for further details on the material presented here.

\subsection{The integral Lichnerowicz formula on spin manifolds with boundary}\label{lichbd}
We assume that the given manifold $(M,g)$ is spin and fix once and for all a spin structure on  $TM$. We denote by $\mathbb SM$ the associated spin bundle and by $\nabla$ both the Levi-Civita connection of $TM$
and its lift to $\mathbb SM$. 
Also,
$c:TM\times\mathbb SM\to\mathbb SM$ is the associated Clifford product, so that the corresponding Dirac operator is 
\[
D=\sum_{i=1}^nc(e_i)\nabla_{e_i},
\]
where $\{e_i\}_{i=1}^n$ is any orthonormal frame.
Sometimes, when we wish to emphasize the dependence of $c$ on $g$ we append a superscript and  write $c=c^g$ instead (and similarly for the other geometric invariants associated to  the given spin structure).

The need to consider chirality boundary conditions for spinors along the boundary leads us to implement a procedure  introduced in \cite{HMR} which allows us to treat the even and odd dimensional cases simultaneously. Given a spin manifold $M$ as above we set $\mathcal EM =\mathbb SM$ if $n$ is even and $\mathcal EM=\mathbb SM\oplus \mathbb SM$ if $n$ is odd. In this latter case, $\mathcal EM$ becomes a Dirac bundle if we define, for a section
\[
\Psi=\left(
\begin{array}{c}
\psi_1\\
\psi_2 
\end{array}
\right)\in \Gamma(\mathcal EM),\quad \psi_i\in \Gamma(\mathbb SM), 
\]  
the Clifford product and connection by
\[
c^{\mathcal E}(X)\Psi=\left(
\begin{array}{c}
c(X)\psi_1\\
-c(X)\psi_2 
\end{array}
\right),\quad \nabla_X^{\mathcal E}\Psi=\left(
\begin{array}{c}
\nabla_X\psi_1\\
\nabla_X\psi_2 
\end{array}
\right), \quad X\in\Gamma(TM).
\]  
As usual, the corresponding Dirac operator is 
\[
D^{\mathcal E}\Psi=\sum_{i=1}^nc^{\mathcal E}(e_i)\nabla^{\mathcal E}_{e_i}\Psi=\left(
\begin{array}{c}
D\psi_1\\
-D\psi_2 
\end{array}
\right).
\]
Finally, in order to unify the notation we set $c^{\mathcal E}=c$, $\nabla^{\mathcal  E}=\nabla$ and $D^{\mathcal E}=D$ if $n$ is even whenever convenient. 

We now define the {\em Killing connections} on  $\mathcal EM$ by
\[
\widetilde\nabla_X^{\mathcal E,\pm}=\nabla^{\mathcal E}_X\pm\frac{\bf i}{2}c^{\mathcal E}(X).
\]
The corresponding {\em Killing-Dirac operators} are defined in the usual way, namely,
\[
\widetilde D^{\mathcal E,\pm}=\sum_{i=1}^nc^{\mathcal E}(e_i)\widetilde\nabla_{e_i}^{\mathcal E,\pm},
\]
so that
\[
\widetilde D^{\mathcal E,\pm}=D^{\mathcal E}\mp\frac{n{\bf i}}{2}.
\]
We remind that $\widetilde D^{\mathcal E,\pm}$ satisfies the fundamental Lichnerowicz formula:
\begin{equation}
\label{lichnefund}
({\widetilde D^{\mathcal E,\pm}})^2=({{{\widetilde\nabla}^{\mathcal E,\pm}}})^*
\widetilde\nabla^{\mathcal E,\pm}+\frac{R_g+n(n-1)}{4}.
\end{equation}

Given  $\Psi\in\Gamma(\mathcal EM)$ we  set
\[
\widetilde\Theta^{\pm}_\Psi(X)=-\langle \widetilde{\mathcal W}^{\mathcal E,\pm}(X)\Psi,\Psi\rangle, \quad X\in \Gamma(TM),
\]
where 
\begin{equation}\label{bdterm}
\widetilde{\mathcal W}^{\mathcal E,\pm}(X)=-(\widetilde\nabla_X^{\mathcal E,\pm}+c^{\mathcal E}(X) \widetilde D^{\mathcal E,\pm}).
\end{equation}
Using (\ref{lichnefund}) we easily compute that
\[
{\rm div}_g\widetilde\Theta^\pm_\Psi=|\widetilde\nabla^{\mathcal E,\pm} \Psi|^2-|\widetilde D^{\mathcal E,\pm}\Psi|^2+\frac{R_g+n(n-1)}{4}|\Psi|^2,
\]
so
if 
$\Omega\subset M$ is a compact domain with a nonempty boundary $\partial\Omega$, which we assume endowed with its inward pointing unit normal $\nu$, then integration by parts yields the integral version of the fundamental Lichnerowicz formula, namely, 
\begin{equation}\label{partshyp}
\int_\Omega\left(|\widetilde\nabla^{\mathcal E,\pm} \Psi|^2-|\widetilde D^{\mathcal E,\pm}\Psi|^2+\frac{R_g+n(n-1)}{4}|\Psi|^2\right)dM={\rm Re}\int_{\partial\Omega}
\left\langle {\widetilde{\mathcal W}}^{\mathcal E,\pm}(\nu)\Psi,\Psi\right\rangle d\Sigma.
\end{equation}

A key step in our argument is to rewrite the right-hand side of (\ref{partshyp}) along the portion of $\partial \Omega$ lying on $\Sigma$ in terms of the corresponding shape operator. Indeed, notice that 
$\Sigma$ carries the bundle $\mathcal EM|_\Sigma$, obtained by restricting  $\mathcal EM$ to $\Sigma$. This becomes a Dirac bundle if endowed with the Clifford product 
$$
c^{\mathcal E,\intercal}(X)\Psi=c^{\mathcal E}(X)c^{\mathcal E}(\nu) \Psi,
$$
and the connection 
\begin{equation}\label{conn0}
\nabla^{\mathcal E,\intercal}_X\Psi  =  \nabla^{\mathcal E}_X\Psi-\frac{1}{2}c^{\mathcal E,\intercal}(AX)\Psi,
\end{equation}
where $A=-\nabla\nu$ is the shape operator of $\Sigma$, so
the corresponding Dirac operator $D^{\mathcal E,\intercal}:\Gamma(\mathcal EM|_\Sigma)\to\Gamma(\mathcal EM|_\Sigma)$ is
$$
D^{\mathcal E,\intercal}\Psi=\sum_{j=1}^{n-1}c^{\mathcal E,\intercal}(f_j)\nabla^{\mathcal E,\intercal}_{f_j}\Psi,
$$
where $\{f_j\}_{j=1}^{n-1}$ is a local orthonormal tangent frame along $\Sigma$.
Choosing the frame so that $Af_j=\kappa_jf_j$, where $\kappa_j$ are the principal curvatures of $\Sigma$, we have
\[
D^{\mathcal E,\intercal}\Psi=-c^{\mathcal E}(\nu)\Dc\Psi+\frac{H}{2}\Psi,
\]
where $H=\kappa_1+\cdots+\kappa_{n-1}$ is the mean curvature and
\[
\Dc\Psi=
\sum_{j=1}^{n-1}c^{\mathcal E}(f_j)\nabla^{\mathcal E}_{f_j}\Psi=c^{\mathcal E}(\nu)\left(D^{\mathcal E,\intercal}\Psi-\frac{H}{2}\Psi\right).
\]
Since $\Dc=D^{\mathcal E}-c^{\mathcal E}(\nu)\nabla^{\mathcal E}_{\nu}$, we obtain
\[
D^{\mathcal E,\intercal}\Psi=\frac{H}{2}\Psi-(\nabla^{\mathcal E}_\nu+c^{\mathcal E}(\nu)D^{\mathcal E})\Psi,
\]
which combined with (\ref{partshyp}) yields the following important result.

\begin{proposition}\label{intpartf} Under the conditions above,
\begin{eqnarray}\label{parts3}
\int_\Omega\left(|\widetilde\nabla^{\mathcal E,\pm} \Psi|^2-|\widetilde D^{\mathcal E,\pm}\Psi|^2+\frac{R_g+n(n-1)}{4}|\Psi|^2\right)d\Omega
& = &\int_{\partial\Omega\cap\Sigma}
\left(\langle \widetilde D^{\mathcal E,\intercal,\pm}\Psi,\Psi\rangle-\frac{H}{2}|\Psi|^2\right) d\Sigma\nonumber \\
& & \quad + {\rm Re}\int_{\partial\Omega\cap {\rm int}\,M}
\left\langle {\widetilde{\mathcal W}}^{\mathcal E,\pm}(\nu)\Psi,\Psi\right\rangle d\partial \Omega,
\end{eqnarray}
where
\begin{equation}\label{newdirac}
\widetilde D^{\mathcal E,\intercal,\pm}=D^{\mathcal E,\intercal}\pm \frac{n-1}{2}{\bf i}c^{\mathcal E}(\nu):\Gamma(\mathbb S{\Sigma})\to
\Gamma(\mathbb S{\Sigma}).
\end{equation}
\end{proposition}

\begin{remark}\label{intrin}
	{\rm It turns out that the extrinsic Dirac bundle $(\mathcal EM|_{\Sigma},c^{\mathcal E,\intercal}, \nabla^{\mathcal E,\intercal})$ can be naturally identified with certain Dirac bundles constructed out of the {\em intrinsic} induced spin bundle $(\mathbb S\Sigma,c^\gamma,\nabla^\gamma)$, where $\gamma=g|_\Sigma$ is the induced metric along $\Sigma$. Thus, 
	\[
	(\mathcal EM|_{\Sigma},c^{\mathcal E,\intercal}, \nabla^{\mathcal E,\intercal})\cong
	\left\{
	\begin{array}{ll}
	(\mathbb S\Sigma\oplus\mathbb S\Sigma,c^\gamma\oplus -c^\gamma,\nabla^\gamma\oplus\nabla^\gamma) &  n\,\,{\rm even}\\
	(\mathbb S\Sigma\oplus\mathbb S\Sigma,c^\gamma\oplus c^\gamma,\nabla^\gamma\oplus\nabla^\gamma) &  n\,\,{\rm odd}.
	\end{array}
		\right.
	\]
Similar remarks hold for the corresponding Dirac operators. 
We refer to \cite[Subsection 2.2]{HMR} for a detailed discussion of these issues. 
}
	\end{remark}

\subsection{Chirality boundary conditions}\label{chirbd}
To further simplify (\ref{parts3}) we must impose suitable boundary conditions on $\Psi$.

\begin{definition}\label{chiralop}
	A {\rm chirality operator}	on a spin manifold $(M,g)$ is a (pointwise) selfadjoint involution $Q:\Gamma(\mathcal EM)\to\Gamma(\mathcal EM)$ which is parallel and anti-commutes with Clifford multiplication by tangent vectors.
\end{definition}

If $n$ is even it is well-known that  Clifford multiplication by the complex volume element
provides a chirality operator. Using the formalism above, we easily see that in case $n$ is odd 
$Q:\Gamma(\mathcal EM)\to \Gamma(\mathcal EM)$,
\[
Q\left(
\begin{array}{c}
\psi_1\\
\psi_2
\end{array}
\right)=\left(
\begin{array}{c}
\psi_2\\
\psi_1
\end{array}
\right)
\]
also defines a chirality operator.	
Thus, in either case we define the corresponding {\em boundary chirality operator} $\mathcal Q=Qc^{\mathcal E}(\nu):\Gamma(\mathcal EM|_{\Sigma})\to\Gamma(\mathcal EM|_{\Sigma})$. Clearly, $\mathcal Q$ is a self-adjoint involution as well, so we may consider the projections
\[
P^{(\pm)}_{\mathcal Q}=\frac{1}{2}\left({\rm Id}_{\mathcal EM|_{\Sigma}}\pm\mathcal Q\right):\Gamma(\mathcal EM|_\Sigma)\to\Gamma(V^{\mathcal E,(\pm)})
\] 
onto the $\pm 1$-eigenbundles $V^{\mathcal E,(\pm)}$ of $\mathcal Q$. Thus, $\Psi\in\Gamma(V^{\mathcal E,(\pm)})$ if and only if $\mathcal Q\Psi=\pm\Psi$. 
In the following we use the qualification {\em standard} to refer to any of these chilarity structures on a spin manifold $M$.

\begin{remark}
	\label{remodd}
{\rm If $n$ is odd then 
$\Psi\in \Gamma(V^{\mathcal E,(\pm)})$ if and only if 
\begin{equation}\label{rep1}
\Psi=\left(
\begin{array}{c}
\psi\\
\pm c (\nu)\psi
\end{array}
\right),\quad \psi\in \Gamma(\mathbb SM|_\Sigma).
\end{equation}}
\end{remark}

For any $\Psi\in\Gamma(\mathcal EM)$ we set $\Psi^{(\pm)}=P^{(\pm)}_{\mathcal Q}\Psi\in\Gamma(V^{\mathcal E,(\pm)})$, so that 
\[
\Psi=\Psi^{(+)}+\Psi^{(-)},
\]
an orthogonal decomposition. Since $D^{\mathcal E,\intercal} c^{\mathcal E}(\nu)=-c^{\mathcal E}(\nu)D^{\mathcal E,\intercal}$ we have
$D^{\mathcal E,\intercal} P^{(\pm)}_{\mathcal Q}=P^{(\mp)}_{\mathcal Q}D^{\mathcal E,\intercal}$, 
and hence
\begin{equation}\label{idpm}
\langle D^{\mathcal E,\intercal} \Psi,\Psi\rangle=\langle D^{\mathcal E,\intercal} \Psi^{(+)},\Psi^{(-)}\rangle+\langle D^{\mathcal E,\intercal} \Psi^{(-)},\Psi^{(+)}\rangle.
\end{equation}

\begin{definition}\label{chirality}
	If $\mathcal EM$ is endowed with the standard chirality operator $Q$ as above then we say that $\Psi\in\Gamma(\mathcal EM)$ satisfies a {\rm chirality boundary condition} if any of the identities
	$\Psi^{(\pm)}=0$ holds along $\Sigma$.
	\end{definition}

\begin{proposition}\label{streg}
	If $\Psi$ satisfies a chirality boundary condition then
	\begin{equation}\label{vanishing0}
	\langle c^{\mathcal E}(\nu)\Psi,\Psi\rangle=0,
	\end{equation}
	and
	\begin{eqnarray}\label{parts4}
	{\rm Re}\int_{\partial\Omega\cap {\rm int}\,M}
	\left\langle {\widetilde{\mathcal W}}^{\mathcal E,\pm}(\nu)\Psi,\Psi\right\rangle d\partial\Omega
	& = & 
	\int_\Omega\left(|\widetilde\nabla^{\mathcal E,\pm} \Psi|^2-|\widetilde D^{\mathcal E,\pm}\Psi|^2+\frac{R_g+n(n-1)}{4}|\Psi|^2\right)d\Omega\nonumber\\
	&  &\quad +\int_{\partial\Omega\cap\Sigma}
	\frac{H}{2}|\Psi|^2 d\Sigma.
	\end{eqnarray} 
\end{proposition}

\begin{proof}
If $\mathcal Q\Psi=\pm\Psi$ we have 
	\begin{eqnarray*}
		\langle c^{\mathcal E}(\nu)\Psi,\Psi\rangle
		& = & \pm \langle c^{\mathcal E}(\nu)Qc^{\mathcal E}(\nu)\Psi,\Psi\rangle\\
		& = & \mp \langle Qc^{\mathcal E}(\nu)\Psi,c^{\mathcal E}(\nu)\Psi\rangle\\
		& = & \mp\langle c^{\mathcal E}(\nu)\Psi,Qc^{\mathcal E}(\nu)\Psi\rangle\\
		& = & - \langle c^{\mathcal E}(\nu)\Psi,\Psi\rangle,
	\end{eqnarray*} 
which proves (\ref{vanishing0}). On the other hand, from (\ref{idpm}) we have  $\langle D^{\mathcal E,\intercal}\Psi,\Psi\rangle=0$. Thus, from (\ref{newdirac}) we get $\langle \widetilde D^{\mathcal E,\intercal,\pm}\Psi,\Psi\rangle=0$, which together with (\ref{parts3}) proves (\ref{parts4}).
\end{proof}

Finally, we remark that the projections $P_{\mathcal Q}^{(\pm)}$ define nice elliptic boundary conditions for the Dirac operador $\widetilde D^{\mathcal E,+}$ considered above, as the following result shows. 

\begin{proposition}\label{existspin}
	If $(M,g,\Sigma)$ is asymptotically hyperbolic as above with $R_g\geq -n(n-1)$ and $H_g\geq 0$ then 
	for any $\Phi\in \Gamma(\mathcal EM)$ such that $\widetilde\nabla^{\mathcal E,+}\Phi\in L^2(\mathcal EM)$ there exists a unique $\Xi\in L_1^2(\mathcal EM)$ solving the boundary value problem
	\[
	\left\{
	\begin{array}{lcc}
	{\widetilde D}^{\mathcal E,+}\Xi=-\widetilde D^{\mathcal E,+}\Phi & {\rm in} & M \\
	\Xi^{(\pm)}=0 & {\rm on}  & \Sigma 
	\end{array} 
	\right.
	\]
\end{proposition}

\begin{proof}
	The assumption $\widetilde \nabla^{\mathcal E,+}\Phi\in L^2(\mathcal EM)$ clearly implies that 
	$\widetilde D^{\mathcal E,+}\Phi\in L^2(\mathcal EM)$ as well.  The result is then an easy  consequence of the methods leading to
	\cite[Corollary 4.19]{GN}.
\end{proof}

\subsection{Killing spinors}\label{killing}

We start by adapting a well-known definition.

\begin{definition}\label{defkill}
	We say that $\Phi\in\Gamma(\mathcal EM)$ is an {\rm imaginary Killing section to the number} $\pm {\bf i}/2$ if it is parallel with respect to $\widetilde \nabla^{\mathcal E,\pm}$, that is,
	\[
	\nabla^{\mathcal E}_X\Phi\pm\frac{\bf i}{2}c^{\mathcal E}(X)\Phi=0,\quad X\in\Gamma(TM).
	\]
	The space of all such sections is denoted by $\mathcal K^{g,\pm}(\mathcal EM)$.
	More generally,
	$\Phi$ is
	{\rm Killing-harmonic} if it satisfies any of the equations $\widetilde D^{\mathcal E,\pm}\Phi=0$.
	\end{definition}

\begin{remark}\label{killimp}
{\rm If $n$ is odd then 
\[
\Phi=\left(
\begin{array}{c}
\phi_+\\
\phi_-
\end{array}
\right)
\]
is imaginary Killing to the number ${\bf i}/2$ if and only if $\phi_{\pm}\in\Gamma(\mathbb SM)$ is imaginary Killing to the number $\pm {\bf i}/2$. Thus,
\[
\mathcal K^{g,+}(\mathcal EM)=
\left\{
\begin{array}{lc}
\mathcal K^{g,+}(\mathbb SM) & n\,\,{\rm even}\\
\mathcal K^{g,+}(\mathbb SM)\oplus\mathcal K^{g,-}(\mathbb SM) & n\,\,{\rm odd}.
\end{array}
\right.
\]
where $\mathcal K^{g,\pm}(\mathbb SM)$ is the space of {imaginary Killing spinors} to the number $\pm {\bf i}/2$.
}
\end{remark}

\begin{example}\label{killhyp}
	{\rm The conformal relationship between $(\mathbb B_+^n,\hat b)$ and $(\mathbb B^n_+,\delta)$ described in (\ref{confrep}) allows us to canonically identify the corresponding spinor bundles $\mathbb S\mathbb B_{+,\hat b}^n$ and $\mathbb S\mathbb B^n_{+,\delta}$, so that $\phi\in\Gamma(\mathbb S\mathbb B^n_{+,\delta})$ corresponds to a certain $\overline{\phi}\in\Gamma(\mathbb S\mathbb B_{+,\hat b}^n)$. Under this identification, if $u\in\Gamma(\mathbb S\mathbb B^n_{+,\delta})$ is a {\em constant}  (i.e. $\nabla^\delta$-parallel) spinor then
		the prescription  
	\begin{equation}\label{presc}
	\phi_{u,\pm}(x)=\omega(x)^{-1/2}\overline{\left(1\pm {\bf{i}} c^{\delta}(x)\right)u}\in\Gamma(\mathbb S\mathbb B_{+,\hat b}^n)
	\end{equation}
	exhausts the space $\mathcal K^{\hat b,\pm}(\mathbb S\mathbb B^n_+)$ \cite{Bau}. Here, $\delta$ refers to the  Euclidean metric. In particular, if $n$ is odd then by Remark \ref{killimp}, 
\begin{equation}\label{rep2}
\Phi_{u,v}=\left(
\begin{array}{c}
	\phi_{u,+}\\
		\phi_{v,-}
\end{array}
\right)\in\mathcal K^{\hat b,+}(\mathcal E\mathbb B^n_+).
\end{equation}
}
	\end{example}

In general, if $(M,g,\Sigma)$ is asymptotically hyperbolic we may consider the space  
\[
\mathcal K^{g,\pm,(\pm)}(\mathcal EM)=\{\Phi\in\mathcal K^{g,\pm}(\mathcal EM);\mathcal Q^g\Phi=\pm\Phi\}
\]
of all imaginary Killing sections to the number $\pm{\bf i}/2$ satisfying the corresponding chirality boundary conditon. 
This space can be explicitly described for $\mathbb H^n_+$.
In view of Example \ref{killhyp} above, it is convenient here to consider the half-disk model $\mathbb B^n_+$   in Remark \ref{ballmodel}. 
If $Q^{\hat b}$ is the standard  chirality operator on $(\mathbb B^n_+,\hat b)$ and $\mathcal Q^{\hat b}$ is the corresponding     
boundary chirality operator, then these data naturally induce corresponding operators $Q^{\delta}$ and $\mathcal Q^\delta$ on $(\mathbb B^n_+,\delta)$. Now let 
$
\mathcal K^{\delta,(\pm)}
$
be the space of all constant spinors $u\in\Gamma(\mathbb S\mathbb B_+^n,\delta)$ which satisfy 
$\mathcal Q^{\delta}u=\pm u$. 

\begin{proposition}\label{isopresc}
If $n$ is even the prescription $u\mapsto \phi_{u,\pm}$ in (\ref{presc}) defines  isomorphisms $\mathcal K^{\delta,(+)}\cong\mathcal K^{\hat b,\pm,(+)}(\mathbb S\mathbb B_n^+)$ and $\mathcal K^{\delta,(-)}\cong\mathcal K^{\hat b,\pm,(-)}(\mathbb S\mathbb B_n^+)$. 
	\end{proposition}

\begin{proof}
Let $\nu_{\hat b}=\omega^{-1}\nu_\delta$ be the hyperbolic unit normal along $\partial \mathbb B^n_+$, where $\nu_\delta=\partial_n$ is the Euclidean inward unit normal.  Notice that $c^\delta(\nu^{\delta})c^\delta(x)=-c^\delta(x)c^\delta(\nu^\delta)$ if $x\in \partial\mathbb B^n_+$. Thus, if $u\in \mathcal K^{\delta,(+)}$, 
\begin{eqnarray*}
	(\mathcal Q^{\hat b}\phi_{u,\pm})(x) 
	& =  & (Q^{\hat b}c^{\hat b}(\nu^{\hat b})\phi_{u,\pm})(x)\\
	& = & \omega(x)^{-1/2}Q^{\hat b}c^{\hat b}(\nu^{\hat b}){(\overline{u}\pm\overline{c^\delta(x)u}})\\
	& = & \omega(x)^{-1/2}\left(\overline{Q^\delta c^\delta(\nu^\delta)u}\pm
	\overline{Q^\delta c^\delta(\nu^\delta)c^\delta(x)u}\right)\\
	& = & \omega(x)^{-1/2}\left(\overline{u}\mp
	\overline{Q^\delta c^\delta(x)c^\delta(\nu^\delta)u}\right)\\
	& = & \omega(x)^{-1/2}\left(\overline{u}\pm
	\overline{c^\delta(x)Q^\delta c^\delta(\nu^\delta)u}\right)\\
	& = & \omega(x)^{-1/2}\left(\overline{u}\pm
	\overline{c^\delta(x)u}\right)\\
	& = & \phi_{u,\pm}(x).
\end{eqnarray*}
Similarly, if $u\in \mathcal K^{\delta,(-)}$ we compute that  $(\mathcal Q^{\hat b}\phi_{u,\pm})(x) =-\phi_{u,\pm}(x)$.
\end{proof}

This leads to the following result, which confirms that $\mathbb H^n_+$ carries the maximal number of linearly independent such sections.

\begin{corollary}\label{dimkill}
	We have
	\[
	\dim_{\mathbb C}\mathcal K^{b,+,(\pm)}(\mathcal E\mathbb H^n_+)=2^{k-1},
	\]
	if $n=2k$ or $n=2k+1$.
	As a consequence,
	\[
		\dim_{\mathbb C}\mathcal K^{b,+}(\mathcal E\mathbb H^n_+)=2^k={\rm rank}\,\mathbb S\mathbb H^n_+.
	\]
\end{corollary}

\begin{proof}
	The even case is obvious since ${\rm rank}\,\mathbb S\mathbb B_+^n=2^k$ equals the dimension of the space of constant spinors. If $n$ is odd note that if $\Psi_{u,v}$ in (\ref{rep2})  is of the form (\ref{rep1}) then 
	\[
	v-c^\delta(x)v=\pm c^\delta(\nu^\delta)(u+c^\delta(x)u),\quad x\in\partial\mathbb B^n_+.
	\]
	By taking $x=0$ yields $v=\pm c^{\delta}(\nu^{\delta})u$, so the entries in (\ref{rep2}) depend on the {\em same} constant spinor. 
\end{proof}

\section{The positive mass theorem and its consequences}\label{mainsec}

In this section we present the proofs of our main results, namely, Theorems \ref{maintheo} and \ref{conjmptheo} below.  
We then explain how they imply Theorems \ref{maintheocor}, \ref{riggen} and  \ref{rigconfcom} in the Introduction.

We consider an asymptotically hyperbolic manifold $(M,g,\Sigma)$ in the sense of Definition \ref{def:as:hyp}. We fix a chart at infinity $F:(\mathbb H^n_{+,r_0},b)\to (M_{\rm ext},g)$. In fact, since $\mathfrak m_{(g,b,F)}$ is an asymptotic invariant, we may assume that $F$ is a {\em global} diffeomorphism between $\mathbb H^n_+$ and $M$. In any case this allows us to construct a gauge map $\mathcal G$ acting on tangent vectors so that 
\begin{equation}\label{impogau}
\langle \mathcal GX,\mathcal GY\rangle_g=\langle X,Y\rangle_b, \quad \langle\mathcal GX,Y\rangle_g=\langle X,\mathcal GY\rangle_g. 
\end{equation}
and such that, in the asymptotic region, 
\begin{equation}\label{asymfr0}
\mathcal G=I-\frac{1}{2}\mathcal H+\mathcal R, \quad \mathcal H= O(r^{-\tau}), \quad \mathcal R=O(r^{-2\tau}). 
\end{equation}
In terms of an orthonormal $b$-frame $\{\mathfrak f_i\}_{i=1}^n$ in Definition \ref{def:as:hyp}, this last requirement means that $\mathfrak e_i=\mathcal G\mathfrak f_i$ given by
\begin{equation}\label{asymfr}
\mathfrak e_i=\mathfrak f_i-\frac{1}{2}\mathcal H\mathfrak f_i+\mathcal R\mathfrak f_i
\end{equation}
is an orthonormal $g$-frame. 
%We have $\mathcal G\mathfrak f_1=\mathfrak f_1$ and since $\mathcal H=I-\mathcal G^2+O(r^{-2\tau})$ we get 
%\begin{equation}\label{normrad1}
%\mathcal H\mathfrak f_1=O(r^{-2\tau}). 
%\end{equation}
Notice that $\mathcal H=I-\mathcal G^2+O(r^{-2\tau})$, so that 
\begin{equation}\label{normrad2}
e(X,Y)=\langle \mathcal HX,Y\rangle_g +O(r^{-2\tau}),
\end{equation}
whenever $X$ and $Y$ are uniformly bounded vector fields. 

The gauge map $\mathcal G$ induces an identification between the bundles $\mathcal E\mathbb H^n_{+,r_0}$ and $\mathcal E M_{\rm ext}$ endowed with the metric structures coming from $b$ and $g$, respectively.  Thus, if $\varphi$ is a cut-off function on $M$ with $\varphi=1$ on $M_{\rm ext}$ and $\Phi\in\Gamma(\mathcal E\mathbb H_{+}^n)$ then $\Phi_*:=\varphi\mathcal G\Phi\in\Gamma(\mathcal EM)$ and the map $\Phi\mapsto\Phi_*$ is a (fiberwise) isometry in a neighborhood of infinity. We apply this construction to $\Phi\in \mathcal K^{b,+,(\pm)}(\mathcal E\mathbb H^n_+)$ an imaginary Killing section 
satisfying a chirality boundary condition; see Example \ref{killhyp} and Corollary \ref{dimkill}.  We set
\[
\mathcal C_b^+=\{V\in\mathcal N_b^+;\langle\langle V,V\rangle\rangle=0, \langle\langle V,V_{(0)}\rangle\rangle\geq 0\}
\]
to be the future-pointing null cone.

\begin{proposition}\label{prelimlem}
For any $\Phi\in \mathcal K^{b,+,(\pm)}(\mathcal E\mathbb H^n_+) $ we have that  $V_\Phi:=\langle\Phi,\Phi\rangle\in\mathcal N_b^+$, and it satisfies $\langle\langle V_{\Phi},V_{\Phi}\rangle\rangle\geq 0$ and $\langle\langle V_{\Phi},V_{(0)}\rangle\rangle\geq 0$. Moreover, every $V\in \mathcal C_b^+$ can be written as $V=V_{\Phi}$ for some $\Phi\in \mathcal K^{b,+,(\pm)}(\mathcal E\mathbb H^n_+) $. In particular, $|\Phi|_b=O(r^{1/2})$.
\end{proposition}
\begin{proof}
	If $n$ is even so that $\Phi={\phi_{u,+}}$,
	direct computations starting from (\ref{presc}) show that for $x'\in\mathbb B^n_+$,
	\[
	V_\Phi(x')=|u|^2V_{(0)}(x')+{\bf i}\sum_{i=1}^n\langle c^\delta(\partial_{x_i'})u,u\rangle V_{(i)}(x')
	\]
	and $\langle\langle V_{\Phi},V_{\Phi}\rangle\rangle=	|u|^4+\sum_{i=1}^n\langle c^\delta(\partial_{x_i'})u,u\rangle^2$. 
	However, the same calculation as that leading to (\ref{vanishing0}) shows that 	$\langle c^\delta(\partial_{x_n'})u,u\rangle=0$. 
	As shown in \cite[Theorem 1]{Bau}, $\langle\langle V_{\Phi},V_{\Phi}\rangle\rangle$ is a nonnegative constant and of course the case $\langle\langle V_{\Phi},V_{\Phi}\rangle\rangle=0$ corresponds to $V_{\Phi}\in\mathcal C_b^+$. 
	Notice that the same conclusion holds if we had taken $\Phi=\phi_{u,-}$ instead. 
	If $n$ is odd, (\ref{rep2}) gives  
	\[
	\langle\Phi,\Phi\rangle=\langle \phi_{u,+},\phi_{u,+}\rangle +\langle \phi_{v,-},\phi_{v,-}\rangle,
	\]
	which also proves the first assertions in those dimensions. Finally, the last assertion follows from the fact that $V=O(r)$.
\end{proof}

We now take a Killing section $\Phi\in \mathcal K^{b,+,(\pm)}(\mathcal E\mathbb H^n_+)$ so that $V_\Phi\in\mathcal N_b^+$ as in Proposition \ref{prelimlem}. We may then extend the transplanted section $\Phi_*$ to the whole of $M$ so that the given chilarity boundary condition is satisfied along $\Sigma$. 
Also, a well-known computation shows that 
\[
|\widetilde \nabla^{\mathcal E,+}\Phi_*|_g\leq C\left(|\mathcal G-I|_b+|\nabla^b\mathcal G|_b\right)|\Phi|_b=O(r^{-\tau+1/2}),
\]
so that $\widetilde \nabla^{\mathcal E,+}\Phi_*\in L^2$ and we may apply Proposition \ref{existspin} to obtain $\Xi\in L^2_1(\mathcal EM)$ such that $\widetilde D^{\mathcal E,+}\Xi=-\widetilde D^{\mathcal E,+}\Phi_*$ and $\Xi^{(\pm)}=0$ along $\Sigma$. Thus, $\Psi_{\Phi}:=\Phi_*+\Xi$ is Killing harmonic ($\widetilde D^{\mathcal E,+}\Psi_{\Phi}=0$) and  $\Psi_{\Phi}^{(\pm)}=0$ along $\Sigma$. Moreover, by  taking into account the identification between $M$ and $\mathbb H^n_+$ given by $F$,  we see that $\Psi_\Phi$ asymptotes $\Phi$ at infinity in the sense that $\Psi_\Phi-\Phi\in L^2_1(M)$. We now state our first main result, which provides a Herzlich-Chru\'sciel-Witten-type formula for the mass functional.

\begin{theorem}\label{maintheo}
	With the notation above,
	\begin{equation}\label{maintheo2}
	\frac{1}{4}\mathfrak m_{(g,b,F)}(V_\Phi)=\int_M\left(|\nabla^{\mathcal E,+}\Psi_\Phi|^2+\frac{R_g+n(n-1)}{4}|\Psi_\Phi|^2\right)dM+\frac{1}{2}\int_\Sigma H_g|\Psi_\Phi|^2d\Sigma,
	\end{equation}
	for any $\Phi\in\mathcal K^{b,+,(\pm)}(\mathcal E\mathbb H^n_+)$.
	\end{theorem}

For the proof we may assume that the chart at infinity is chosen as in Remark \ref{massbterm}. After  using (\ref{parts4}) with $\Omega=M_r$, the region in $M$ bounded by $\Sigma_r\cup S_{r,+}^{n-1}$, and $\Psi=\Psi_{\Phi}$ we get
	\begin{eqnarray*}
{\rm Re}\int_{S_{r,+}^{n-1}}
\left\langle {\widetilde{\nabla}}_\mu^{\mathcal E,+}\Psi_\Phi,\Psi_\Phi\right\rangle dS_{r,+}^{n-1}
& = & 
\int_{M_r}\left(|\widetilde\nabla^{\mathcal E,+} \Psi_\Phi|^2+\frac{R_g+n(n-1)}{4}|\Psi_\Phi|^2\right)dM_r\nonumber\\
&  &\quad +\int_{\Sigma_r}
\frac{H}{2}|\Psi_\Phi|^2 d\Sigma,
\end{eqnarray*}
and using that the second term in the right-hand of  (\ref{massdef2}) does not contribute to the mass, 
we need to check that 
\begin{equation}\label{checkf}
	\lim_{r\to +\infty}{\rm Re}\int_{S_{r,+}^{n-1}}
	\left\langle {\widetilde{\nabla}}_\mu^{\mathcal E,+}\Psi_\Phi,\Psi_\Phi\right\rangle dS_{r,+}^{n-1}=\lim_{r\to+\infty}\frac{1}{4}\int_{S_{r,+}^{n-1}}\langle\mathbb U(V_{\Phi},e),\mu\rangle dS_{r,+}^{n-1}.
\end{equation}
In fact, after splitting the integrand on the left-hand side by means of the decomposition $\Psi_\Phi=\Phi_*+\Xi$, we see that algebraic cancellations and the decaying properties of $\widetilde\nabla^{\mathcal E,+}\Phi_*$ and $\Xi$ imply that 
\[
	\lim_{r\to +\infty}{\rm Re}\int_{S_{r,+}^{n-1}}
\left\langle {\widetilde{\nabla}}_\mu^{\mathcal E,+}\Psi_\Phi,\Psi_\Phi\right\rangle dS_{r,+}^{n-1}=
	\lim_{r\to +\infty}{\rm Re}\int_{S_{r,+}^{n-1}}
\left\langle {\widetilde{\nabla}}_\mu^{\mathcal E,+}\Xi,\Phi_*\right\rangle dS_{r,+}^{n-1},
\] 	
so we shall focus on the last integrand.

To proceed we follow \cite{LP} and introduce the $(n-2)$-form
\[
\varepsilon=\left\langle[\mathfrak e_l\cdot,\mathfrak e_m\cdot]\Phi_*,\Xi\right\rangle \mathfrak e_l\righthalfcup\mathfrak e_m\righthalfcup dM,
\]
where for simplicity we denote the Clifford multiplication  by a dot.

\begin{lemma}\label{prelimlem2}
	We have 
	\[
	\lim_{r\to +\infty}\int_{S^{n-2}_r}\varepsilon=0.
	\] 
\end{lemma}

\begin{proof}
	It follows from (\ref{asymfr}) that 
	$\mathfrak e_i \righthalfcup\mathfrak e_j=\mathfrak f_i\righthalfcup\mathfrak f_j+O(r^{-\tau})$, so if we again take $\mathfrak f$ as in Remark \ref{massbterm} and use that $\mathfrak f_n \righthalfcup\mathfrak f_\alpha\righthalfcup dM=dS^{n-2}_r$ we have that, restricted to the boundary, 
	\[
	\varepsilon=-4\langle\mathfrak e_\alpha \cdot\mathfrak e_n\cdot\Phi_*,\Xi\rangle dS^{n-2}_r+\langle O(r^{-\tau})\cdot\Phi_*,\Xi\rangle dS^{n-2}_r.
	\] 
	Using that $dS^{n-2}_r=O(r^{n-2})$, $|\Phi_*|=O(r^{1/2})$ and the decaying properties of $\Xi$, we see that the last term integrates to zero at infinity. 
	On the other hand, recalling that $\mathfrak e_n=-\nu$ and  both $\Phi_*$ and $\Xi$ satisfy the chirality boundary conditions along $\Sigma$, 
	the remaining integrand equals
	\begin{eqnarray*}
		4\langle \mathfrak e_\alpha\cdot\nu \cdot\Phi_*,\Xi\rangle & = & 4\langle\mathfrak e_\alpha\cdot\nu\cdot(\pm Q\nu\cdot\Phi_*),\pm Q\nu\cdot\Xi\rangle\\
		& = & -4\langle Q\mathfrak e_\alpha\cdot\Phi_*,Q\nu\cdot\Xi\rangle\\
		& = & -4\langle\mathfrak e_\alpha\cdot\Phi_*,\nu\cdot\Xi\rangle,
	\end{eqnarray*}  
	so that 
	\begin{equation}\label{compeq}
	\langle \mathfrak e_\alpha\cdot\nu\cdot\Phi_*,\Xi\rangle=
	\langle \nu\cdot\mathfrak e_\alpha\cdot\Phi_*,\Xi\rangle.
	\end{equation}
	From this and  
	Clifford relations we get
	\begin{eqnarray*}
		\langle \mathfrak e_\alpha\cdot\nu \cdot\Phi_*,\Xi\rangle & = & 
		\frac{1}{2}\langle (\mathfrak e_\alpha\cdot\nu \cdot+\nu\cdot\mathfrak e_\alpha\cdot)\Phi_*,\Xi\rangle\\
		& = & -\langle\mathfrak e_\alpha,\nu\rangle_g\langle\Phi_*,\Xi\rangle\\
		& = & 0,
	\end{eqnarray*}
	which completes the proof.
\end{proof}

A straightforward computation gives 
\[
d\varepsilon=4\left(\langle \widetilde{\mathcal W}^{\mathcal E,+}(\mathfrak e_l)\Phi_*,\Xi\rangle-\langle \Phi_*, \widetilde{\mathcal W}^{\mathcal E,+}(\mathfrak e_l)\Xi\rangle\right){\mathfrak e_l}\righthalfcup dM,
\]
where by (\ref{bdterm}),
\begin{equation}\label{decayc}
 \widetilde{\mathcal W}^{\mathcal E,+}(\mathfrak e_l)=-(\delta_{lm}+\mathfrak e_l\cdot\mathfrak e_m\cdot)\widetilde\nabla^{\mathcal E,+}_{\mathfrak e_m}=-\frac{1}{2}[\mathfrak e_l\cdot,\mathfrak e_m\cdot]\widetilde\nabla^{\mathcal E,+}_{\mathfrak e_m}.
\end{equation}
Hence,
\[
\int_{S_{r,+}^{n-1}}\langle  \widetilde{\mathcal W}^{\mathcal E,+}(\mathfrak e_l)\Phi_*,\Xi\rangle {\mathfrak e_l}\righthalfcup dM-\int_{S_{r,+}^{n-1}}\langle  \widetilde{\mathcal W}^{\mathcal E,+}(\mathfrak e_l)\Xi,\Phi_*\rangle {\mathfrak e_l}\righthalfcup dM = \frac{1}{4}\int_{S_{r}^{n-2}}\varepsilon,
\]
where we used that
$S_r^{n-2}=\partial S_{r,+}^{n-1}$.   Thus,  
\begin{eqnarray*}
	{\rm Re}\int_{S_{r,+}^{n-1}}\left\langle \widetilde\nabla^{\mathcal E,+}_{\mathfrak e_l}\Xi,\Phi_*\right\rangle{\mathfrak e_l}\righthalfcup dM
	 & = & - {\rm Re}\int_{S_{r,+}^{n-1}}\left\langle (\widetilde{\mathcal W}^{\mathcal E,+}(\mathfrak e_l)-\mathfrak e_l\cdot \widetilde D^{\mathcal E,+})\Xi,\Phi_*\right\rangle{\mathfrak e_l}\righthalfcup dM\\
	& = & - {\rm Re}\int_{S_{r,+}^{n-1}}\left\langle \widetilde{\mathcal W}^{\mathcal E,+}(\mathfrak e_l)\Phi_*,\Xi\right\rangle{\mathfrak e_l}\righthalfcup dM + \frac{1}{4}{\rm Re}\int_{S_r^{n-2}}\varepsilon \\
	& & \quad + {\rm Re}\int_{S_{r,+}^{n-1}}\left\langle \mathfrak e_l\cdot \widetilde D^{\mathcal E,+}\Xi,\Phi_*\right\rangle{\mathfrak e_l}\righthalfcup dM\\
	& = & 
	-{\rm Re}\int_{S_{r,+}^{n-1}}\left\langle \widetilde{\mathcal W}^{\mathcal E,+}(\mathfrak e_l)\Phi_*,\Xi\right\rangle{\mathfrak e_l}\righthalfcup dM + \frac{1}{4}{\rm Re}\int_{S_r^{n-2}}\varepsilon \\
	& & \quad  - {\rm Re}\int_{S_{r,+}^{n-1}}\left\langle \mathfrak e_l\cdot \widetilde D^{\mathcal E,+}\Phi_*,\Phi_*\right\rangle{\mathfrak e_l}\righthalfcup dM,
	\end{eqnarray*}
where we used that $\Psi_\Phi$ is Killing harmonic in the last step. 
Again due to the decay properties and (\ref{decayc}), the first integral in the right-hand side above vanishes at infinity, so recalling  that $\mathfrak e_i=\mathfrak f_i+O(r^{-\tau})$ we finally
obtain
\[
\lim_{r\to+\infty}	{\rm Re}\int_{S_{r,+}^{n-1}}\langle\widetilde{\nabla}_\mu^{\mathcal E,+}\Xi,\Phi_*\rangle dS_{r,+}^{n-1}  = 
-\lim_{r\to +\infty}{\rm Re}\int_{S_{r,+}^{n-1}}\left\langle \mu\cdot \widetilde D^{\mathcal E,+}\Phi_*,\Phi_*\right\rangle d{S_{r,+}^{n-1}},
	\]
where we used Lemma \ref{prelimlem2} to get rid of the integral involving $\varepsilon$.
Now, a standard computation as in \cite[Section 4]{CH} shows that
\[
\lim_{r\to +\infty}{\rm Re}\int_{S_{r,+}^{n-1}}\left\langle \mu\cdot \widetilde D^{\mathcal E,+}\Phi_*,\Phi_*\right\rangle d{S_{r,+}^{n-1}}=-
\lim_{r\to +\infty}\frac{1}{4}\int_{S_{r,+}^{n-1}}\langle \mathbb U(V_\Phi,e),\mu\rangle dS_{r,+}^{n-1},
\]
which proves (\ref{checkf})
and completes the proof of Theorem \ref{maintheo}.

We now explain how the mass formula (\ref{maintheo2}) implies Conjecture \ref{conjmp} in case $M$ is spin. More precisely, we have the following result.

\begin{theorem}\label{conjmptheo}
	Let $(M,g,\Sigma)$ be an asymptotically hyperbolic spin manifold with $R_g\geq -n(n-1)$ and $H_g\geq 0$. Then for  any admissible chart $F$ the mass vector $\mathcal P^{[F]}$ is time-like  future directed, unless it vanishes and $(M,g,\Sigma)$ is isometric to $(\mathbb H^n_{+},b,\partial \mathbb H^n_+)$.
\end{theorem}

\begin{proof}
From (\ref{maintheo2}) and the assumptions $R_g\geq -n(n-1)$ and $H_g\geq 0$ we see that 
\begin{equation}\label{forany}
\langle\langle\mathcal P^{[F]},V\rangle\rangle\geq 0
\end{equation}
for any $V=V_{\Phi}$ with $\Phi\in \mathcal K^{b,+,(\pm)}(\mathcal E\mathbb H^n_+)$. In particular, (\ref{forany}) holds for any $V\in\mathcal C_b^+$ since, by Proposition \ref{prelimlem}, any such $V$ can be written as $V=V_{\Phi}$ for some $\Phi\in \mathcal K^{b,+,(\pm)}(\mathcal E\mathbb H^n_+)$.
Hence, $\mathcal P^{[F]}$ is time-like and future directed  unless there exists some $V=V_\Phi\neq 0$ so that the equality holds in (\ref{forany}). This last possibility implies 
by (\ref{maintheo2}) that there exists  a non-trivial Killing section $\Psi=\Psi_\Phi$ on $M$ satisfying the corresponding chirality boundary condition along $\Sigma$. In particular, $g$ is Einstein with ${\rm Ric}_g=-(n-1)g$ so that $\widehat E_g=0$ as in Remark \ref{remeins}.  
 On the other hand, for any $X\in\Gamma(T\Sigma)$ we have, upon derivation of $\mathcal Q\Psi=\pm\Psi$, 
\[
Q c^{\mathcal E}(\nabla_X\nu)\Psi
=  -Qc^{\mathcal E}(\nu)\widetilde\nabla_X^{\mathcal E,+}\Psi\pm\widetilde\nabla_X^{\mathcal E,+}\Psi=0,
\]
and since $\Psi$ never vanishes we see that 
$\nabla_X\nu=0$, that is, $\Sigma$ is totally geodesic.
Thus, by our alternate definition (\ref{asymhypl2}) of the mass functional we see that $\mathcal P^{[F]}$ vanishes and hence the equality in (\ref{forany}) holds for {\em any} $V=V_{\Phi}$ with $\Phi\in  \mathcal K^{b,+,(\pm)}(\mathcal E\mathbb H^n_+)$.
By Corollary \ref{dimkill} and (\ref{maintheo2}), this means that $(M,g)$ carries as many Killing sections to the number ${\bf i}/2$ (i.e. parallel sections for the connection $\widetilde\nabla^{\mathcal E,+}$) as the reference space $(\mathbb H^n_+,b,\partial\mathbb H^n_+)$, which implies that $g$ is locally hyperbolic. Moreover, by (\ref{conn0}), Remark \ref{intrin} and Corollary \ref{dimkill}, the restrictions of these Killing sections to $\Sigma$ generate a space of imaginary Killing spinors on  $\mathbb S \Sigma$ with maximal dimension. In particular, this implies that $\Sigma$ has no compact components and so it is connected by Remark \ref{rmk:bd}. Also, a direct application of Gauss formula shows that $R_\gamma+(n-1)(n-2)=0$ so $(\Sigma,\gamma)$ is asymptotically hyperbolic as a boundaryless manifold (in the sense of \cite{CH}). A well-known result by H. Baum \cite{Bau} implies that  $(\Sigma,\gamma)$ is isometric to $(\mathbb H^{n-1},\beta)$. 
Now, we double the manifold $(M, g, \Sigma)$ along $\Sigma$ obtaining a locally hyperbolic manifold $(\widehat M, \widehat g)$  which is asymptotically hyperbolic as a boundaryless manifold. Standard topological arguments show that $(\widehat M, \widehat g)$ is isometric to $(\mathbb H^n, b)$ and finally $(M,g,\Sigma)$ is isometric to $(\mathbb H^n_+, b,  \partial\mathbb H^n_+)$.
\end{proof}

Clearly, Theorem \ref{maintheocor} is an immediate consequence of Theorem \ref{conjmptheo}.
Also, Theorem \ref{riggen} follows promptly from (\ref{asymhypl2}), Remark \ref{remeins} and Theorem \ref{conjmptheo}, and similarly for Theorem \ref{riggennbd}, which follows from the corresponding boundaryless statements.
We next present the proof of Theorem \ref{rigconfcom}. Since $g$ is Einstein, if $n=3$ the result follows from Theorem \ref{maintheocor}. If $n\geq 4$ the well-known computation in \cite{AD} shows that the contribution for the mass functional coming from integration over the asymptotic hemisphere $S^{n-1}_+$ vanishes. Thus, by (\ref{massdef2}) and the fact that $t
\approx r^{-1}$, it remains to check that the asymptotic integral over $S^{n-2}=\partial S^{n-1}_+$ vanishes as well, that is,
\[
\lim_{t\to 0}\int_{S^{n-2}_t}Ve_{1 n}dS^{n-2}_t= 0, 
\] 
where $V\in\mathcal N_b^+$ and we have chosen the frame $\mathfrak f$  so that  $\mathfrak f_1=\sinh t\,\partial_t=\vartheta$ and $\mathfrak f_n=-\eta$. Now recall that we can set up a gauge map $\mathcal G$ in the asymptotic region so that 
\[
\mathcal G\mathfrak f_i=\mathfrak f_i-\frac{1}{2}\mathcal H\mathfrak f_i+\mathcal R\mathfrak f_i, 
\]
is an orthonormal $g$-frame, 
where $\mathcal H=O(t^{n-2})$ and $\mathcal R=o(t^{n-1})$ by (\ref{remainder}). 
But notice that  (\ref{expan}) clearly leads to $\mathcal G\mathfrak f_1=\mathfrak f_1$ so we actually have  $\mathcal H\mathfrak f_1=o(t^{n-1})$. Thus, by the analogue of (\ref{normrad2}),
\[
	e_{1n}  =  \langle\mathcal H\mathfrak f_1,\mathfrak f_n\rangle_g +o(t^{n-1})=o(t^{n-1}).
\]	
Since $V=O(t^{-1})$ and $dS^{n-2}_t=O(t^{2-n})$,  
\[
\int_{S^{n-2}_t}Ve_{1 n}dS^{n-2}_t= o(1), 
\] 
which proves the claim and finishes the proof of Theorem \ref{rigconfcom}.


\begin{thebibliography}{999999} 

	\bibitem[ABdL]{ABdL} S. Almaraz, E. Barbosa, and L. L. de Lima,
	A positive mass theorem for asymptotically flat manifolds with a non-compact boundary. {\em Commun. Anal. Geom.} {\bf 24}(4) (2016), 673-715.
	
	\bibitem[An]{An}  
	M. T. Anderson, 
	Geometric aspects of the AdS/CFT correspondence. {\em  AdS/CFT correspondence: Einstein metrics and their conformal boundaries}, 1-31, 
	IRMA Lect. Math. Theor. Phys., 8, Eur. Math. Soc., Z\"urich, 2005.
	
	\bibitem[ACG]{ACG}
	L.~Andersson, M.~Cai, and G.~J. Galloway.
	\newblock Rigidity and positivity of mass for asymptotically hyperbolic
	manifolds.
	\newblock {\em Ann. Henri Poincar\'e}, 9(1):1--33, 2008.
	
	\bibitem[AD]{AD}
	L.~Andersson and M.~Dahl.
	\newblock Scalar curvature rigidity for asymptotically locally hyperbolic
	manifolds.
	\newblock {\em Ann. Global Anal. Geom.}, 16(1):1--27, 1998.
	
	
	
	\bibitem[Bar]{Bar} R. Bartnik, 
	The mass of an asymptotically flat manifold. 
	{\em Commun. Pure and Appl. Math.}
	{\bf 39} (1986), 661-693.
	
\bibitem[Bau]{Bau} H. Baum, Complete Riemannian manifolds with imaginary Killing spinors, 
	{\em Ann. Glob. Anal. Geom.} {\bf 7} (1989), 205--226.


\bibitem[BCN]{BCN} H. Barzegar, P. T. Chru\'sciel, L. Nguyen. On the total mass of asymptotically hyperbolic manifolds, {\em arXiv:1812.03924}.


	\bibitem[BM]{BM} S. Brendle, F. C. Marques, Recent progress on the Yamabe problem. {\em Surveys in geometric analysis and relativity}, 29–47, Adv. 	Lect. Math. (ALM), 20, Int. Press, Somerville, MA, 2011.	
	
	\bibitem[CH]{CH} 
	P. T. Chru\'sciel  and M. Herzlich,
	The mass of asymptotically hyperbolic Riemannian manifolds.
	{\em Pacific J. Math.} {\bf{212}} (2003), 231-264.
	
	
	\bibitem[CLW]{CLW} X. Chen, M. Lai, F. Wang,  
	Escobar-Yamabe compactifications for Poincar\'e-Einstein manifolds and rigidity theorems, {\em arXiv:1712.02540}.
	
	
	\bibitem[CN]{CN}
	P. T. Chru{\'s}ciel and G.~Nagy.
	\newblock The mass of spacelike hypersurfaces in asymptotically anti-de
	{S}itter space-times.
	\newblock {\em Adv. Theor. Math. Phys.}, 5(4):697--754, 2001.
	
	\bibitem[DGS]{DGS} M. Dahl, R. Gicquaud, and A. Sakovich. Penrose type inequalities for asymptotically
	hyperbolic graphs. {\em Ann. Henri Poincar\'e}, 14(5):1135-1168, 2013.
	
	\bibitem[dLG]{dLG} L.L. de Lima, F. Gir\~ao, An Alexandrov–Fenchel-type inequality in hyperbolic space with an application to a Penrose
	inequality, Ann. Henri Poincar\'e 17 (2016) 979–1002. 
	
	\bibitem[dLGM]{dLGM} L. L. de Lima, F. Gir\~ao, A. Montalb\'an, The mass in terms of Einstein and Newton, {\em arXiv:1811.06924}.
	
	\bibitem[dLPZ]{dLPZ} L.  L. de Lima, P. Piccione, M. Zedda,  On bifurcation of solutions of the Yamabe problem in product manifolds. {\em Ann. Inst. H. Poincar\'e Anal. Non Lin\'aire} 29 (2012), no. 2, 261-277.

	\bibitem[GN]{GN} N. Grosse, R. Nakad,
	Boundary value problems for noncompact boundaries of Spin$^c$ manifolds and spectral estimates,
	{\em Proc. Lond. Math. Soc.} (3) 109 (2014), no. 4, 946-974.
	
	\bibitem[HM]{HM} O. Hijazi, S. Montiel, A holographic principle for the existence of parallel spinor fields and an inequality of Shi-Tam type,
	{\em Asian J. Math.} 18 (2014) 489-506.
	
	\bibitem[HMR]{HMR} O. Hijazi, S. Montiel, S. Raulot, A holographic principle for the existence of imaginary Killing spinors, {\em J. Geom. Phys.} 91
	(2015) 12-28.
	
	
	
	\bibitem[He1]{He1}
	M.~Herzlich.
	\newblock Mass formulae for asymptotically hyperbolic manifolds.
	\newblock In {\em Ad{S}/{CFT} correspondence: {E}instein metrics and their
		conformal boundaries}, volume~8 of {\em IRMA Lect. Math. Theor. Phys.}, pages
	103--121. Eur. Math. Soc., Z\"urich, 2005.
	
		\bibitem[He2]{He2}
	M. Herzlich,
	Computing asymptotic invariants with the Ricci tensor on asymptotically flat and asymptotically hyperbolic manifolds.  
	{\em Ann. Henri Poincar\'e} 17 (2016), no. 12, 3605-3617.
	
	\bibitem[LP]{LP} J. M.  Lee, T. H. Parker, The Yamabe problem. {\em Bull. Amer. Math. Soc. (N.S.)} 17 (1987), no. 1, 37-91.
	
	\bibitem[LQS]{LQS}  G. Li, J. Qing, Y. Shi, Gap phenomena and curvature estimates for conformally compact
	Einstein manifolds, {\em Trans. Amer. Math. Soc.} 369 (2017), no. 6, 4385-4413.
	
	\bibitem[Lo]{Lo} J. Lohkamp,
	The higher dimensional positive mass theorem II, {\em arXiv:1612.07505}.
	
	
	\bibitem[Ma]{Ma} D. Maerten, Positive energy-momentum theorem for AdS-asymptotically hyperbolic manifolds. {\em Ann. Henri Poincar\'e} 7 (2006), no. 5, 975-1011.
	
	\bibitem[MP]{MP} R. Mazzeo, F. Pacard, Constant curvature foliations in asymptotically hyperbolic
	spaces. {\em Rev. Mat. Iberoam.} 27 (2011), no. 1, 303-333.
	
	\bibitem[Mi]{Mi}
	B.~Michel.
	\newblock Geometric invariance of mass-like asymptotic invariants.
	\newblock {\em J. Math. Phys.}, 52(5):052504, 14, 2011.
	
	\bibitem[M-O]{M-O}
	M.~Min-Oo.
	\newblock Scalar curvature rigidity of asymptotically hyperbolic spin
	manifolds.
	\newblock {\em Math. Ann.}, 285(4):527--539, 1989.
	
	\bibitem[Ra]{Ra} S. Raulot,
	A remark on the rigidity of conformally compact Poincar\'e-Einstein manifolds, {\em arXiv  arXiv:1803.03162}.
	
	
	\bibitem[Sc]{Sc} R. Schoen, Conformal deformation of a Riemannian metric to constant scalar curvature. {\em J. Differential Geom.} 20 (1984), no. 2, 479-495. 
	
	\bibitem[SY1]{SY1}
	R. Schoen, S.-T. Yau,
	On the proof of the positive mass conjecture in General Relativity. 
	{\em Comm. Math. Phys.} {\bf{65}} (1979), 45-76.
	
	\bibitem[SY2]{SY2} R. Schoen, S.-T. Yau,
	Positive Scalar curvature and minimal hypersurface singularities, {\em arXiv:1704.05490}.
	
	\bibitem[Wa]{Wa}
	X. Wang, 
	Mass for asymptotically hyperbolic manifolds. 
	{\em J. Diff. Geom}. {\bf 57} (2001), 273-299.
	
	\bibitem[Wi]{Wi}
	E. Witten,
	A new proof of the positive energy theorem. 
	  {\em Comm. Math. Phys.} {\bf{80}} (1981), 381-402. 
	
	
\end{thebibliography}
\end{document}